\documentclass[12pt]{amsart}
\usepackage{mathrsfs}
\usepackage{etex}       
\usepackage{txfonts}
\usepackage{amssymb}
\usepackage{diagbox}
\usepackage{eucal}
\usepackage{graphicx}
\usepackage{amsmath}
\usepackage{amscd}
\usepackage[all]{xy}           
\usepackage{tikz,color,xcolor}
\usepackage{amsfonts,latexsym}
\usepackage{xspace}
\usepackage{epsfig}
\usepackage{float}
\usepackage{fancybox}
\usepackage{multicol}
\usepackage{colordvi}
\usepackage{pstcol,pstricks} 
\usepackage{makecell}
\usepackage{booktabs}
\usepackage{cancel}
\usepackage{pstricks-add} 
\usepackage[active]{srcltx} 
\ifpdf
\usepackage[colorlinks,final,backref=page,hyperindex]{hyperref}
\else
\usepackage[colorlinks,final,backref=page,hyperindex,hypertex]{hyperref}
\fi

\newtheorem{theorem}{Theorem}[section]
\newtheorem{prop}[theorem]{Proposition}
\newtheorem{lemma}[theorem]{Lemma}
\newtheorem{coro}[theorem]{Corollary}

\theoremstyle{definition}
\newtheorem{defn}[theorem]{Definition}

\newtheorem{exam}[theorem]{Example}

\newtheorem{prop-def}[theorem]{Proposition-Definition}
\newtheorem{coro-def}[theorem]{Corollary-Definition}

\newcommand{\nc}{\newcommand}
\nc{\tred}[1]{\textcolor{red}{#1}}
\nc{\tblue}[1]{\textcolor{blue}{#1}}
\nc{\tgreen}[1]{\textcolor{green}{#1}}
\nc{\tpurple}[1]{\textcolor{purple}{#1}}
\nc{\btred}[1]{\textcolor{red}{\bf #1}}
\nc{\btblue}[1]{\textcolor{blue}{\bf #1}}
\nc{\btgreen}[1]{\textcolor{green}{\bf #1}}
\nc{\btpurple}[1]{\textcolor{purple}{\bf #1}}
\nc{\NN}{{\mathbb N}}
\nc{\ncsha}{{\mbox{\cyr X}^{\mathrm NC}}} \nc{\ncshao}{{\mbox{\cyr
X}^{\mathrm NC}_0}}

\nc{\vsa}{\vspace{-.1cm}} \nc{\vsb}{\vspace{-.2cm}}
\nc{\vsc}{\vspace{-.3cm}} \nc{\vsd}{\vspace{-.4cm}}
\nc{\vse}{\vspace{-.5cm}}

\renewcommand{\textbf}[1]{}

\newcommand{\delete}[1]{}

\nc{\mlabel}[1]{\label{#1}}  
\nc{\mcite}[1]{\cite{#1}}  
\nc{\mref}[1]{\ref{#1}}  
\nc{\meqref}[1]{\eqref{#1}}  
\nc{\mbibitem}[1]{\bibitem{#1}} 

\delete{
\nc{\mlabel}[1]{\label{#1}  
{\hfill \hspace{1cm}{\tt{{\ }\hfill(#1)}}}}
\nc{\mcite}[1]{\cite{#1}{{\tt{{\ }(#1)}}}}  
\nc{\mref}[1]{\ref{#1}{{\tt{{\ }(#1)}}}}  
\nc{\meqref}[1]{\eqref{#1}{{\tt{{\ }(#1)}}}}  
\nc{\mbibitem}[1]{\bibitem[\bf #1]{#1}} 
}

\nc{\opa}{\ast} \nc{\opb}{\odot} \nc{\op}{\bullet} \nc{\pa}{\frakL}
\nc{\arr}{\rightarrow} \nc{\lu}[1]{(#1)} \nc{\mult}{\mrm{mult}}
\nc{\diff}{\mathfrak{Diff}}
\nc{\opc}{\sharp}\nc{\opd}{\natural}
\nc{\ope}{\circ}
\nc{\dpt}{\mathrm{d}}
\nc{\tforall}{\text{ for all }}
\nc{\diam}{alternating\xspace}
\nc{\Diam}{Alternating\xspace}
\nc{\cdiam}{alternating\xspace}
\nc{\Cdiam}{Alternating\xspace}
\nc{\AW}{\mathcal{A}}
\nc{\rba}{Rota-Baxter algebra\xspace}

\nc{\ari}{\mathrm{ar}}

\nc{\lef}{\mathrm{lef}}

\nc{\Sh}{\mathrm{ST}}

\nc{\Cr}{\mathrm{Cr}}

\nc{\st}{{Schr\"oder tree}\xspace}
\nc{\sts}{{Schr\"oder trees}\xspace}

\nc{\vertset}{\Omega} 

\delete{
\nc{\assop}{\quad \begin{picture}(5,5)(0,0)
\line(-1,1){10}
\put(-2.2,-2.2){$\bullet$}
\line(0,-1){10}\line(1,1){10}
\end{picture} \quad \smallskip}

\nc{\operator}{\begin{picture}(5,5)(0,0)
\line(0,-1){6}
\put(-2.6,-1.8){$\bullet$}
\line(0,1){9}
\end{picture}}

\nc{\idx}{\begin{picture}(6,6)(-3,-3)
\put(0,0){\line(0,1){6}}
\put(0,0){\line(0,-1){6}}
 \end{picture}}
}

\nc{\pb}{{\mathrm{pb}}}
\nc{\Lf}{{\mathrm{Lf}}}

\nc{\lft}{{left tree}\xspace}
\nc{\lfts}{{left trees}\xspace}

\nc{\fat}{{fundamental averaging tree}\xspace}

\nc{\fats}{{fundamental averaging trees}\xspace}
\nc{\avt}{\mathrm{Avt}}

\nc{\rass}{{\mathit{RAss}}}

\nc{\aass}{{\mathit{AAss}}}

\nc{\vin}{{\mathrm Vin}}    
\nc{\lin}{{\mathrm Lin}}    
\nc{\inv}{\mathrm{I}n}
\nc{\gensp}{V} 
\nc{\genbas}{\mathcal{V}} 
\nc{\bvp}{V_P}     
\nc{\gop}{{\,\omega\,}}     

\nc{\bin}[2]{ (_{\stackrel{\scs{#1}}{\scs{#2}}})}  
\nc{\binc}[2]{ \left (\!\! \begin{array}{c} \scs{#1}\\
    \scs{#2} \end{array}\!\! \right )}  
\nc{\bincc}[2]{  \left ( {\scs{#1} \atop
    \vspace{-1cm}\scs{#2}} \right )}  
\nc{\bs}{\bar{S}} \nc{\cosum}{\sqsubset} \nc{\la}{\longrightarrow}
\nc{\rar}{\rightarrow} \nc{\dar}{\downarrow} \nc{\dprod}{**}
\nc{\dap}[1]{\downarrow \rlap{$\scriptstyle{#1}$}}
\nc{\md}{\mathrm{dth}} \nc{\uap}[1]{\uparrow
\rlap{$\scriptstyle{#1}$}} \nc{\defeq}{\stackrel{\rm def}{=}}
\nc{\disp}[1]{\displaystyle{#1}} \nc{\dotcup}{\
\displaystyle{\bigcup^\bullet}\ } \nc{\gzeta}{\bar{\zeta}}
\nc{\hcm}{\ \hat{,}\ } \nc{\hts}{\hat{\otimes}}
\nc{\barot}{{\otimes}} \nc{\free}[1]{\bar{#1}}
\nc{\uni}[1]{\tilde{#1}} \nc{\hcirc}{\hat{\circ}} \nc{\lleft}{[}
\nc{\lright}{]} \nc{\lc}{\lfloor} \nc{\rc}{\rfloor}
\nc{\curlyl}{\left \{ \begin{array}{c} {} \\ {} \end{array}
    \right .  \!\!\!\!\!\!\!}
\nc{\curlyr}{ \!\!\!\!\!\!\!
    \left . \begin{array}{c} {} \\ {} \end{array}
    \right \} }
\nc{\longmid}{\left | \begin{array}{c} {} \\ {} \end{array}
    \right . \!\!\!\!\!\!\!}
\nc{\onetree}{\bullet} \nc{\ora}[1]{\stackrel{#1}{\rar}}
\nc{\ola}[1]{\stackrel{#1}{\la}}
\nc{\ot}{\otimes} \nc{\mot}{{{\boxtimes\,}}}
\nc{\otm}{\overline{\boxtimes}} \nc{\sprod}{\bullet}
\nc{\scs}[1]{\scriptstyle{#1}} \nc{\mrm}[1]{{\rm #1}}
\nc{\margin}[1]{\marginpar{\rm #1}}   
\nc{\dirlim}{\displaystyle{\lim_{\longrightarrow}}\,}
\nc{\invlim}{\displaystyle{\lim_{\longleftarrow}}\,}
\nc{\mvp}{\vspace{0.3cm}} \nc{\tk}{^{(k)}} \nc{\tp}{^\prime}
\nc{\ttp}{^{\prime\prime}} \nc{\svp}{\vspace{2cm}}
\nc{\vp}{\vspace{8cm}} \nc{\proofbegin}{\noindent{\bf Proof: }}
\nc{\proofend}{$\blacksquare$ \vspace{0.3cm}}
\nc{\modg}[1]{\!<\!\!{#1}\!\!>}
\nc{\intg}[1]{F_C(#1)} \nc{\lmodg}{\!
<\!\!} \nc{\rmodg}{\!\!>\!}
\nc{\cpi}{\widehat{\Pi}}
\nc{\sha}{{\mbox{\cyr X}}}  
\nc{\shap}{{\mbox{\cyrs X}}} 
\nc{\shan}{{\overrightarrow \sha}}
\nc{\shpr}{\diamond}    
\nc{\shp}{\ast} \nc{\shplus}{\shpr^+}
\nc{\shprc}{\shpr_c}    
\nc{\msh}{\ast} \nc{\zprod}{m_0} \nc{\oprod}{m_1}
\nc{\vep}{\varepsilon} \nc{\labs}{\mid\!} \nc{\rabs}{\!\mid}
\nc{\sqmon}[1]{\langle #1\rangle}

\nc{\mmbox}[1]{\mbox{\ #1\ }} \nc{\dep}{\mrm{dep}} \nc{\fp}{\mrm{FP}}
\nc{\rchar}{\mrm{char}} \nc{\End}{\mrm{End}} \nc{\Fil}{\mrm{Fil}}
\nc{\Mor}{Mor\xspace} \nc{\gmzvs}{gMZV\xspace}
\nc{\gmzv}{gMZV\xspace} \nc{\mzv}{MZV\xspace}
\nc{\mzvs}{MZVs\xspace} \nc{\Hom}{\mrm{Hom}} \nc{\id}{\mrm{id}}
\nc{\im}{\mrm{im}} \nc{\incl}{\mrm{incl}} \nc{\map}{\mrm{Map}}
\nc{\mchar}{\rm char} \nc{\nz}{\rm NZ} \nc{\supp}{\mathrm Supp}

\nc{\Alg}{\mathbf{Alg}} \nc{\Bax}{\mathbf{Bax}} \nc{\bff}{\mathbf f}
\nc{\bfk}{{\bf k}} \nc{\bfone}{{\bf 1}} \nc{\bfx}{\mathbf x}
\nc{\bfy}{\mathbf y}
\nc{\base}[1]{\bfone^{\otimes ({#1}+1)}} 
\nc{\Cat}{\mathbf{Cat}}

\nc{\detail}{\marginpar{\bf More detail}
    \noindent{\bf Need more detail!}
    \svp}
\nc{\Int}{\mathbf{Int}} \nc{\Mon}{\mathbf{Mon}}
\nc{\rbtm}{{shuffle }} \nc{\rbto}{{Rota-Baxter }}
\nc{\remarks}{\noindent{\bf Remarks: }} \nc{\Rings}{\mathbf{Rings}}
\nc{\Sets}{\mathbf{Sets}} \nc{\wtot}{\widetilde{\odot}}
\nc{\wast}{\widetilde{\ast}} \nc{\bodot}{\bar{\odot}}
\nc{\bast}{\bar{\ast}} \nc{\hodot}[1]{\odot^{#1}}
\nc{\hast}[1]{\ast^{#1}} \nc{\mal}{\mathcal{O}}
\nc{\tet}{\tilde{\ast}} \nc{\teot}{\tilde{\odot}}
\nc{\oex}{\overline{x}} \nc{\oey}{\overline{y}}
\nc{\oez}{\overline{z}} \nc{\oef}{\overline{f}}
\nc{\oea}{\overline{a}} \nc{\oeb}{\overline{b}}
\nc{\weast}[1]{\widetilde{\ast}^{#1}}
\nc{\weodot}[1]{\widetilde{\odot}^{#1}} \nc{\hstar}[1]{\star^{#1}}
\nc{\lae}{\langle} \nc{\rae}{\rangle}
\nc{\lf}{\lfloor}
\nc{\rf}{\rfloor}

\nc{\QQ}{{\mathbb Q}}
\nc{\RR}{{\mathbb R}} \nc{\ZZ}{{\mathbb Z}}
\nc{\CC}{{\mathbb C}}

\nc{\cala}{{\mathcal A}} \nc{\calb}{{\mathcal B}}
\nc{\calc}{{\mathcal C}}
\nc{\cald}{{\mathcal D}} \nc{\cale}{{\mathcal E}}
\nc{\calf}{{\mathcal F}} \nc{\calg}{{\mathcal G}}
\nc{\calh}{{\mathcal H}} \nc{\cali}{{\mathcal I}}
\nc{\call}{{\mathcal L}} \nc{\calm}{{\mathcal M}}
\nc{\caln}{{\mathcal N}}\nc{\calo}{{\mathcal O}}
\nc{\calp}{{\mathcal P}} \nc{\calr}{{\mathcal R}}
\nc{\cals}{{\mathcal S}} \nc{\calt}{{\mathcal T}}
\nc{\calu}{{\mathcal U}} \nc{\calw}{{\mathcal W}} \nc{\calk}{{\mathcal K}}
\nc{\calx}{{\mathcal X}} \nc{\CA}{\mathcal{A}}

\nc{\fraka}{{\mathfrak a}} \nc{\frakA}{{\mathfrak A}}
\nc{\frakb}{{\mathfrak b}} \nc{\frakB}{{\mathfrak B}}
\nc{\frakD}{{\mathfrak D}} \nc{\frakF}{\mathfrak{F}}
\nc{\frakf}{{\mathfrak f}} \nc{\frakg}{{\mathfrak g}}
\nc{\frakH}{{\mathfrak H}} \nc{\frakL}{{\mathfrak L}}
\nc{\frakM}{{\mathfrak M}} \nc{\bfrakM}{\overline{\frakM}}
\nc{\frakm}{{\mathfrak m}} \nc{\frakP}{{\mathfrak P}}
\nc{\frakN}{{\mathfrak N}} \nc{\frakp}{{\mathfrak p}}
\nc{\frakS}{{\mathfrak S}} \nc{\frakT}{\mathfrak{T}}
\nc{\frakX}{{\mathfrak X}} \nc{\frakx}{\mathfrak{x}}
\nc{\frakc}{{\mathfrak c}}
\nc{\frakd}{{\mathfrak d}}
\nc{\BS}{\mathbb{S}}

\font\cyr=wncyr10 \font\cyrs=wncyr7
\nc{\li}[1]{\textcolor{red}{#1}}
\nc{\lir}[1]{\textcolor{red}{Li:#1}}

\nc{\sz}[1]{\textcolor{blue}{SZ: #1}}
\nc{\qhz}[1]{\textcolor{orange}{QHZ: #1}}

\nc{\ID}{\mathfrak{I}} \nc{\lbar}[1]{\overline{#1}}
\nc{\bre}{{\rm b}} \nc{\sd}{\cals} \nc{\rb}{\rm RB}
\nc{\A}{\rm angularly decorated\xspace} \nc{\LL}{\rm L}
\nc{\w}{\rm wid} \nc{\arro}[1]{#1}
\nc{\ver}{\rm ver}
\nc{\dd}{\diamond}
\nc{\dr}{\diamond_r}
\nc{\dg}{\diamond_e}
\nc{\dk}{\diamond_\bfk}
\nc{\shar}{{\mbox{\cyrs X}}_r} 
\nc{\shag}{{\mbox{\cyrs X}}_g}
\nc{\de}{\Delta}
\nc{\delg}{\Delta_e}
\nc{\da}{\Delta_A}
\nc{\dgg}{\Delta_g}
\nc{\va}{\vep_A }
\nc{\ve}{\vep }
\nc{\vg}{\vep_e }
\nc{\vgg}{\vep_g }
\nc{\bug}{\bullet_e}
\nc{\bt}{\bar{\ot}}
\nc{\pg}{P_e}
\nc{\mg}{\mu_e}
\nc{\fs}{\frakS}

\nc{\wmrb}{extended Rota-Baxter algebra\xspace}
\nc{\wmrbs}{extended Rota-Baxter algebras\xspace}
\nc{\wmrbo}{extended Rota-Baxter operator\xspace}
\nc{\wmrbos}{extended Rota-Baxter operators\xspace}
\nc{\Wmrbos}{Extended Rota-Baxter operators\xspace}
\nc{\wmrbi}{extended Rota-Baxter identity\xspace}

\nc\dashed{\begin{pspicture}(0.2,0.3)
\psline[linecolor=blue,linestyle=dashed](0,-0.1)(0,0.3)
\end{pspicture}}

\nc\dlpa{\begin{pspicture}(0.65,0.65)
\psset{xunit=15pt,yunit=15pt}\psgrid[subgriddiv=1,griddots=4,
gridlabels=4pt](0,0)(1,1)
\psline(0,0)(0,1)
\psline(0,1)(1,1)
\end{pspicture}}
\def\dlpb{
\begin{pspicture}(0.65,0.65)
\psset{xunit=15pt,yunit=15pt}\psgrid[subgriddiv=1,griddots=4,
gridlabels=4pt](0,0)(1,1)
\psline(0,0)(1,0)
\psline(1,0)(1,1)
\end{pspicture}}
\nc\dlpc{\begin{pspicture}(0.65,0.65)
\psset{xunit=15pt,yunit=15pt}\psgrid[subgriddiv=1,griddots=4,
gridlabels=4pt](0,0)(1,1)
\psline[linecolor=red](0,0)(1,1)
\end{pspicture}}
\def\dlpd{
\begin{pspicture}(0.65,0.65)
\psset{xunit=15pt,yunit=15pt}\psgrid[subgriddiv=1,griddots=4,
gridlabels=4pt](0,0)(1,1)
\psline[linecolor=blue,linestyle=dashed](0,0)(1,1)
\end{pspicture}}
\nc\dlpe{\centering\begin{pspicture}(1,1)
\psset{xunit=15pt,yunit=15pt}\psgrid[subgriddiv=1,griddots=4,
gridlabels=4pt](0,0)(1,2)
\psline[linecolor=black](0,0)(1,0)
\psline[linecolor=black](1,0)(1,2)
\end{pspicture}
}
\nc\dlpee{\centering\begin{pspicture}(1,1)
\psset{xunit=15pt,yunit=15pt}\psgrid[subgriddiv=1,griddots=4,
gridlabels=4pt](0,0)(1,2)
\psline[linecolor=black](0,0)(0,2)
\psline[linecolor=black](0,2)(1,2)
\end{pspicture}
}

\nc\dlpf{\begin{pspicture}(1,1)
\psset{xunit=15pt,yunit=15pt}\psgrid[subgriddiv=1,griddots=4,
gridlabels=4pt](0,0)(1,2)
\psline[linecolor=black](0,0)(0,1)
\psline[linecolor=black](0,1)(1,1)
\psline[linecolor=black](1,1)(1,2)
\end{pspicture}
}

\nc\dlpg{
\begin{pspicture}(1,1)
\psset{xunit=15pt,yunit=15pt}\psgrid[subgriddiv=1,griddots=4,
gridlabels=4pt](0,0)(1,2)
\psline[linecolor=red](0,0)(1,1)
\psline[linecolor=black](1,1)(1,2)
\end{pspicture}}
\nc\dlpgg{
\begin{pspicture}(1,1)
\psset{xunit=15pt,yunit=15pt}\psgrid[subgriddiv=1,griddots=4,
gridlabels=4pt](0,0)(1,2)
\psline[linecolor=black](0,0)(0,1)
\psline[linecolor=red](0,1)(1,2)
\end{pspicture}}
\nc\dlph{
\begin{pspicture}(1,1)
\psset{xunit=15pt,yunit=15pt}\psgrid[subgriddiv=1,griddots=4,
gridlabels=4pt](0,0)(1,2)
\psline[linecolor=blue,linestyle=dashed](0,0)(1,1)
\psline[linecolor=black](1,1)(1,2)
\end{pspicture}
}
\nc\dlphh{
\begin{pspicture}(1,1)
\psset{xunit=15pt,yunit=15pt}\psgrid[subgriddiv=1,griddots=4,
gridlabels=4pt](0,0)(1,2)
\psline[linecolor=black](0,0)(0,1)
\psline[linecolor=blue,linestyle=dashed](0,1)(1,2)
\end{pspicture}
}
\nc\dlphr{\centering\begin{pspicture}(1,2)
\psset{xunit=15pt,yunit=15pt}\psgrid[subgriddiv=1,griddots=4,
gridlabels=4pt](0,0)(1,3)
\psline[linecolor=black](0,0)(1,0)
\psline[linecolor=black](1,0)(1,3)
\end{pspicture}
}
\nc\dlphs{\centering\begin{pspicture}(1,2)
\psset{xunit=15pt,yunit=15pt}\psgrid[subgriddiv=1,griddots=4,
gridlabels=4pt](0,0)(1,3)
\psline[linecolor=black](0,0)(0,1)
\psline[linecolor=black](0,1)(1,1)
\psline[linecolor=black](1,1)(1,3)
\end{pspicture}
}

\nc\dlpht{\begin{pspicture}(1,2)
\psset{xunit=15pt,yunit=15pt}\psgrid[subgriddiv=1,griddots=4,
gridlabels=4pt](0,0)(1,3)
\psline[linecolor=black](0,0)(0,2)
\psline[linecolor=black](0,2)(1,2)
\psline[linecolor=black](1,2)(1,3)
\end{pspicture}
}

\nc\dlphu{
\begin{pspicture}(1,2)
\psset{xunit=15pt,yunit=15pt}\psgrid[subgriddiv=1,griddots=4,
gridlabels=4pt](0,0)(1,3)
\psline[linecolor=black](0,0)(0,3)
\psline[linecolor=black](0,3)(1,3)
\end{pspicture}}

\nc\dlphv{
\begin{pspicture}(1,2)
\psset{xunit=15pt,yunit=15pt}\psgrid[subgriddiv=1,griddots=4,
gridlabels=4pt](0,0)(1,3)
\psline[linecolor=red](0,0)(1,1)
\psline[linecolor=black](1,1)(1,3)
\end{pspicture}}
\nc\dlphw{
\begin{pspicture}(1,2)
\psset{xunit=15pt,yunit=15pt}\psgrid[subgriddiv=1,griddots=4,
gridlabels=4pt](0,0)(1,3)
\psline[linecolor=blue,linestyle=dashed](0,0)(1,1)
\psline[linecolor=black](1,1)(1,3)
\end{pspicture}
}
\nc\dlphx{
\begin{pspicture}(1,2)
\psset{xunit=15pt,yunit=15pt}\psgrid[subgriddiv=1,griddots=4,
gridlabels=4pt](0,0)(1,3)
\psline[linecolor=black](0,0)(0,1)
\psline[linecolor=red](0,1)(1,2)
\psline[linecolor=black](1,2)(1,3)
\end{pspicture}
}
\nc\dlphy{
\begin{pspicture}(1,2)
\psset{xunit=15pt,yunit=15pt}\psgrid[subgriddiv=1,griddots=4,
gridlabels=4pt](0,0)(1,3)
\psline[linecolor=black](0,0)(0,1)
\psline[linecolor=blue,linestyle=dashed](0,1)(1,2)
\psline[linecolor=black](1,2)(1,3)
\end{pspicture}}
\nc\dlphz{
\begin{pspicture}(1,2)
\psset{xunit=15pt,yunit=15pt}\psgrid[subgriddiv=1,griddots=4,
gridlabels=4pt](0,0)(1,3)
\psline[linecolor=black](0,0)(0,2)
\psline[linecolor=blue,linestyle=dashed](0,2)(1,3)
\end{pspicture}
}
\nc\dlphzz{
\begin{pspicture}(1,2)
\psset{xunit=15pt,yunit=15pt}\psgrid[subgriddiv=1,griddots=4,
gridlabels=4pt](0,0)(1,3)
\psline[linecolor=black](0,0)(0,2)
\psline[linecolor=red](0,2)(1,3)
\end{pspicture}
}

\nc\dlpi{
\begin{pspicture}(1,1)
\psset{xunit=15pt,yunit=15pt}\psgrid[subgriddiv=1,griddots=4,
gridlabels=4pt](0,0)(2,2)
\psline(0,0)(0,1)
\psline(0,1)(1,1)
\psline(1,1)(1,2)
\psline(1,2)(2,2)
\end{pspicture}}
\nc\dlpj{
\begin{pspicture}(1,1)
\psset{xunit=15pt,yunit=15pt}\psgrid[subgriddiv=1,griddots=4,
gridlabels=4pt](0,0)(2,2)
\psline(0,0)(1,0)
\psline(1,0)(1,2)
\psline(1,2)(2,2)
\end{pspicture}}
\nc\dlpk{
\begin{pspicture}(1,1)
\psset{xunit=15pt,yunit=15pt}\psgrid[subgriddiv=1,griddots=4,
gridlabels=4pt](0,0)(2,2)
\psline[linecolor=red](0,0)(1,1)
\psline(1,1)(1,2)
\psline(1,2)(2,2)
\end{pspicture}}
\nc\dlpl{
\begin{pspicture}(1,1)
\psset{xunit=15pt,yunit=15pt}\psgrid[subgriddiv=1,griddots=4,
gridlabels=4pt](0,0)(2,2)
\psline[linecolor=blue](0,0)(1,1)
\psline(1,1)(1,2)
\psline(1,2)(2,2)
\end{pspicture}}
\nc\dlpm{
\begin{pspicture}(1,1)
\psset{xunit=15pt,yunit=15pt}\psgrid[subgriddiv=1,griddots=4,
gridlabels=4pt](0,0)(2,2)
\psline[linecolor=red](0,0)(1,1)
\psline[linecolor=blue](1,1)(2,2)
\end{pspicture}}
\nc\dlpn{
\begin{pspicture}(1,1)
\psset{xunit=15pt,yunit=15pt}\psgrid[subgriddiv=1,griddots=4,
gridlabels=4pt](0,0)(2,2)
\psline[linecolor=blue](0,0)(1,1)
\psline[linecolor=blue](1,1)(2,2)
\end{pspicture}
}
\nc\dlpo{
\begin{pspicture}(1,2)
\psset{xunit=15pt,yunit=15pt}\psgrid[subgriddiv=1,griddots=4,
gridlabels=4pt](0,0)(2,3)
\psline(0,0)(0,1)
\psline(0,1)(0,2)
\psline(0,2)(1,2)
\psline(1,2)(2,2)
\psline(2,2)(2,3)
\end{pspicture}
}
\nc\dlpp{
\begin{pspicture}(1,2)
\psset{xunit=15pt,yunit=15pt}\psgrid[subgriddiv=1,griddots=4,
gridlabels=4pt](0,0)(2,3)
\psline(0,0)(1,0)
\psline(1,0)(1,2)
\psline(1,2)(2,2)
\psline(2,2)(2,3)
\end{pspicture}}
\nc\dlpq{
\begin{pspicture}(1,2)
\psset{xunit=15pt,yunit=15pt}\psgrid[subgriddiv=1,griddots=4,
gridlabels=4pt](0,0)(2,3)
\psline(0,0)(1,0)
\psline(1,0)(1,1)
\psline[linecolor=blue](1,1)(2,2)
\psline(2,2)(2,3)
\end{pspicture}}
\nc\dlpr{
\begin{pspicture}(1,2)
\psset{xunit=15pt,yunit=15pt}\psgrid[subgriddiv=1,griddots=4,
gridlabels=4pt](0,0)(2,3)
\psline(0,0)(0,1)
\psline(0,1)(1,1)
\psline[linecolor=red](1,1)(2,2)
\psline(2,2)(2,3)
\end{pspicture}}

\nc\dlps{
\begin{pspicture}(1,2)
\psset{xunit=15pt,yunit=15pt}\psgrid[subgriddiv=1,griddots=4,
gridlabels=4pt](0,0)(2,3)
\psline[linecolor=red](0,0)(1,1)
\psline(1,1)(1,2)
\psline[linecolor=blue](1,2)(2,3)
\end{pspicture}}
\nc\dlpt{
\begin{pspicture}(1,2)
\psset{xunit=15pt,yunit=15pt}\psgrid[subgriddiv=1,griddots=4,
gridlabels=4pt](0,0)(2,3)
\psline[linecolor=blue](0,0)(1,1)
\psline(1,1)(1,2)
\psline[linecolor=red](1,2)(2,3)
\end{pspicture}}

\definecolor{darkred}{rgb}{0.7,0,0} 

\definecolor{darkgreen}{RGB}{0,180,0}

\begin{document}

\title[Extended Rota-Baxter algebras, colored Delannoy paths and Hopf algebras]{Extended Rota-Baxter algebras, diagonally colored Delannoy paths and Hopf algebras}
%

\author{Shanghua Zheng}
\address{School of Mathematics and Statistics, Jiangxi Normal University, Nanchang, Jiangxi 330022, China}
\email{zhengsh@jxnu.edu.cn}

\author{Li Guo}
\address{Department of Mathematics and Computer Science,
	Rutgers University,
	Newark, NJ 07102, USA}
\email{liguo@rutgers.edu}

\author{Huizhen Qiu}
\address{School of Mathematics and Statistics, Jiangxi Normal University, Nanchang, Jiangxi 330022, China}
\email{1197147595@qq.com}

\date{\today}
\begin{abstract}
The Rota-Baxter operator and the modified Rota-Baxter operator on various algebras are both important in mathematics and mathematical physics. The former is originated from the integration-by-parts formula and probability with applications to the renormalization of quantum field theory and the classical Yang-Baxter equation. The latter originated from Hilbert transformations with applications to ergodic theory and the modified Yang-Baxter equation.
Their merged form, called the \wmrbo, has also found interesting applications recently.

This paper presents a systematic study of the \wmrbo. We show that \wmrbos have properties similar to Rota-Baxter operatos and provide a linear structure that unifies Rota-Baxter operators and modified Rota-Baxter operators.
Examples of \wmrbos are also given, especially from polynomials and Laurent series due to their importance in ($q$-)integration and the renormalization in quantum field theory.
We then construct free commutative \wmrbs by a generalization of the quasi-shuffle product. The multiplication of the initial object in the category of commutative \wmrbs allows a combinatorial interpretation in terms of a color-enrichment of Delannoy paths.
Applying its universal property, we equip a free commutative \wmrb with a coproduct which has a cocycle condition, yielding a bialgebraic structure. We then show that this bialgebra on a free \wmrb possesses an increasing filtration and a connectedness property, culminating at a Hopf algebraic structure on a free commutative \wmrb.
\end{abstract}

\makeatletter
\@namedef{subjclassname@2020}{\textup{2020} Mathematics Subject Classification}
\makeatother
\subjclass[2020]{
16W99, 
17B37, 
16S10, 
16T10, 
16T30,  
57R56
}

\keywords{Rota-Baxter algebra, modified Rota-Baxter algebra, cocycle, classical Yang-Baxter equation, modified Yang-Baxter equation, Delannoy path, shuffle product, bialgebra, Hopf algebra}

\maketitle

\vspace{-1cm}

\tableofcontents

\vspace{-1cm}

\setcounter{section}{0}

\allowdisplaybreaks

\section{Introduction}
Rota-Baxter algebras and modified Rota-Baxter algebras both have their origins from integrals in probability and ergodic theory~\mcite{Ba,Co,Tri}. Interestingly, their analogs for Lie algebras are both related to the classical Yang-Baxter equation~\mcite{STS}. The purpose of this paper is to put the two classical linear operators in the same context.

\vsb
\subsection{Rota-Baxter operators and the classical Yang-Baxter equation}

For a given scalar $\lambda$, a {\bf Rota-Baxter algebra of weight $\lambda$} is a pair $(R,P)$
consisting of an associative algebra $R$ and a linear operator $P:R\to R$ that satisfies the {\bf Rota-Baxter identity}
\begin{equation}
P(x)P(y)=P\big(xP(y)\big)+P\big(P(x)y\big)+\lambda P(xy), \ \quad x,y \in R.
\mlabel{eq:rb}
\end{equation}
Then $P$ is called a {\bf Rota-Baxter operator (RBO) of weight $\lambda$}.
The notion of a Rota-Baxter algebra can be viewed as an algebraic abstraction of  the algebra of continuous functions  equipped with the Riemann integral operator $I[f](x):=\int_a^xf(t)\,dt$, as an analog of a differential algebra which was explored initially by Ritt and Kolchin~\mcite{Ko,PS,Ri} as an algebraic study of differential equations. In fact, taking $\lambda=0$, then the Rota-Baxter identity recovers the integration-by-parts formula.

Rota-Baxter algebra was introduced by G. Baxter to understand Spitzer's identity in fluctuation theory in probability~\mcite{Ba}.  Afterwards,  it attracted attentions of prominent mathematicians such as Atkinson, Cartier and Rota~\mcite{Ca,Ro,Ro2}.  In the early 1980s,
Rota-Baxter operators on Lie algebras were independently discovered by Semenov-Tian-Shansky in~\mcite{STS} as the operator form of the classical Yang-Baxter equation~\mcite{BD}, named after the physicists C. N. Yang and R. J. Baxter.

Over the past two decades, the study of associative Rota-Baxter algebras has undergone an extraordinary renascence thanks to broad applications and connections such as renormalization of quantum field theory~\mcite{CK00}, Yang-Baxter equations~\mcite{Bai,BG,Bor,GS2,GG}, operads~\mcite{Ag,BBGN},
combinatorics~\mcite{YGT}, Lie groups~\mcite{GLS21,LS22}, deformation theories~\mcite{TBGS19} and Hopf algebras~\mcite{AGKO,ZGG,ZGG19}.
See~\mcite{Gub} for further details.
\vsb
\subsection{Modified Rota-Baxter operators and the modified Yang-Baxter equation}
The term modified Rota-Baxter operator stemmed from the notion of the modified classical Yang-Baxter equation, which was also introduced in the above-mentioned work of Semenov-Tian-Shansky~\mcite{STS} as a modification of the operator form of the classical Yang-Baxter equation.  It was later applied  to the studies of nonabelian generalized Lax pairs, affine geometry of Lie groups and (extended) $\calo$-operators~\mcite{BGN,Bor}.

However, as already observed by Rota and Smith~\mcite{RS}, the associative analog of the modified classical Yang-Baxter equation had appeared much earlier in the work of the prominent analyst Tricomi in 1951~\mcite{Tri} and then of Cotlar in his remarkable 1955 work~\mcite{Co} unifying Hilbert transformations and ergodic theory. See also~\mcite{Pe,Roo,RS}. If a linear map $Q:R\to R$ satisfies the {\bf modified Rota-Baxter identity} (with a fixed scalar $\kappa$)
\begin{equation}\mlabel{eq:mrb}
Q(x)Q(y)=Q\big(xQ(y)\big)+Q\big(Q(x)y\big)+\kappa xy, \ \quad x,y \in R.
\end{equation}
Then $Q$ is called a {\bf modified Rota-Baxter operator (MRBO) of weight $\kappa$}, and the pair $(R,Q)$  is called a {\bf modified Rota-Baxter algebra of weight $\kappa$}. When $\kappa=-1$, then Eq.~(\mref{eq:mrb}) for a Lie algebra corresponds to the modified Yang-Baxter equation.

\vsb
\subsection{\Wmrbos}

Due to the importance of Rota-Baxter algebras and modified Rota-Baxter algebras, it is natural to study their common expansions. Indeed, for their Lie algebra counterpart, such structures have been applied to study Lax pairs in integrable systems and to give a unified study of tensor forms and operator forms of the classical Yang-Baxter equation~\mcite{BGN,BGN3}. The associative constructions have also be  given~\mcite{BGN2}, under the term extended $\calo$-operators, in order to study associative Yang-Baxter equations.

Motivated by these applications and the recent developments on the related algebraic structures~\mcite{BGN2,ZGG,ZGG19,ZGG192}, this paper presents a systematic algebraic study of the combination of Rota-Baxter algebras and modified Rota-Baxter algebras, called {\bf \wmrbs}.

Several interesting properties discovered from this investigation further justify their study.

First, the linear closedness is a desired property of any algebraic structure. The related studies can be found from the classical notions of bi-Hamiltonian systems~\mcite{Mag} and compatible Lie algebras~\mcite{GS1} to the recent multi-pre-Lie algebras~\mcite{BHZ} and matching Rota-Baxter algebras~\mcite{ZyGG}. Extended Rota-Baxter operators are shown to have this property under certain compatibility conditions (Theorem~\mref{thm:genfij}), providing a linear system containing both Rota-Baxter operators and modified Rota-Baxter operators.

Second, the free construction of \wmrbs inherits the combinatorial significance of the Rota-Baxter algebras and modified Rota-Baxter algebras, including the shuffle products and Delannoy paths (Theorem~\,\mref{thm:ab} and Proposition~\,\mref{pp:del}).

Finally, the free construction also has a natural connected bialgebra and thus a Hopf algebra structure, induced by a cocycle condition (Theorem~\,\ref{thm:rbhopf}).

This investigation also sheds further light to the recent studies of Rota's program in classifying operator identities satisfied by linear operators on algebras, in analog to the much studied polynomial identities (PI) of algebras~\mcite{DF,Pr}.

Rota in~\mcite{Ro2} observed that linear operators satisfying algebraic identities (derivation and Rota-Baxter operators for example) were essential in studying diverse phenomena in mathematics. He then suggested a program to classify such linear operators. Following Rota's suggestion, Rota's Classification Problem of linear operators was formulated and several progresses on the problems were made, leading to newly discovered operators such as differential type operators and Rota-Baxter type operators~\mcite{BE22,GGZs,GSZ,ZGG23,ZGGS}. The \wmrbo belongs to neither of these operators, suggesting a new direction in studying Rota's program.

In this work we mostly focus on commutative \wmrbs. Further study of  the noncommutative case will be continued in a separate work.

\vspace{-.45cm}
\subsection{Layout of the paper}
The paper is organized as follows.

In Section~\mref{sec:gps}, we show that many properties of Rota-Baxter operators can be generalized to \wmrbos. Furthermore, the lack of linear closedness of Rota-Baxter operators and modified Rota-Baxter operators is remedied by \wmrbos. This property provides a rich class of \wmrbs.

Section~\mref{sec:FreeCommGRBA} constructs free \wmrbos generated by a commutative algebra $A$. Motivated by the construction of free commutative (modified) Rota-Baxter algebras in ~\mcite{GK1,ZGG19}, we first establish a product on the tensor module generated by $A$ in terms of a generalized quasi-shuffle product. We then define an \wmrbo on the tensor module to be the right-shift operator to obtain the free \wmrb on $A$. As a special case, the free \wmrb on the base ring $\bfk$ is provided at the end, which gives a generalization of the divided power algebra and has a combinatorial interpretation by colored Delannoy paths.

Another source of inspiration of this study comes from the Hopf algebra structures established on free commutative (modified) Rota-Baxter algebras. Thus Section~\mref{sec: Hopfcomm} aims to equip free commutative \wmrbs with a bialgebra structure, where the comultiplication and the counit are constructed by the universal property of free commutative \wmrbs and a cocycle property.  Furthermore, when the generating algebra $A$ is a connected filtered bialgebra, the bialgebra structure on the free commutative \wmrbs is also connected and hence a Hopf algebra.
We end the section by giving an explicit construction of the Hopf algebra in the special case of $A=\bfk$.
\smallskip

\noindent
{\bf Convention.} Throughout this paper,  $\bfk$ is a commutative unitary ring.  All algebras are taken to be over $\bfk$, as are the linear maps and tensor products, unless otherwise specified.

\vsb
\section{Properties and examples of \wmrbs}
\mlabel{sec:gps}
In this section, we first give the notion of an \wmrb that generalizes the Rota-Baxter algebra and the modified Rota-Baxter algebra. Then we give examples and general properties of \wmrbs, emphasizing on its role in providing a linear structure that contains both  Rota-Baxter operators and modified Rota-Baxter operators.

\begin{defn}
Let $R$ be a $\bfk$-algebra and $\lambda,\kappa \in \bfk$. A linear map $P:R\to R$ is called an {\bf \wmrbo of (double) weight $(\lambda,\kappa)$}, or simply an {\bf \wmrbo} if $P$ satisfies the {\bf \wmrbi}
\begin{equation}
	P(x)P(y)=P\big(xP(y)\big)+P\big(P(x)y\big)+\lambda P(xy)+\kappa xy , \quad  x,y\in R. 
\mlabel{eq:grb}
\end{equation}
Then the pair $(R,P)$ is called an {\bf \wmrb of (double) weight $(\lambda,\kappa)$}, or simply an {\bf \wmrb}.
An {\bf \wmrb homomorphism} $f: (R_1,P_1)\to (R_2,P_2)$ between two \wmrbs $(R_1,P_1)$ and  $(R_2,P_2)$ of the same weight $(\lambda,\kappa)$ is an algebra homomorphism $f: R_1\to R_2$ such that $f\circ P_1 = P_2\circ f$.
\mlabel{defn:freeGRBA}
\vsb
\end{defn}

In the relative context, the \wmrbo is a part of the extended $\calo$-operators (or extended relative Rota-Baxter operators)~\mcite{BGN3}. Since the notion requires some background to define and is not needed in this paper, we refer the reader to the original literature for details and applications to Yang-Baxter equations. As noted in the introduction, our goal is to give a systematic study of this algebraic structure, together with their combinatorial significance.

A basic example of Rota-Baxter operator of weight $\lambda$ is the scalar product $P:=-\lambda \id:R\to R $ ~\cite[Exercise 1.1.7]{Gub}. Further the weight of a Rota-Baxter operator is unique except in degenerated cases. Here is a more general result for \wmrbs.
\begin{prop}
Let $P:R\to R$ be a linear operator.
\begin{enumerate}
\item
If $P=\mu \id$ is a scalar product for $\mu\in\bfk$, then $P$ is an \wmrbo of weight $(\lambda,\kappa)$ precisely for $(\lambda,\kappa)=(\lambda,-\mu^2-\mu\lambda)$ with $\lambda\in \bfk$ arbitrary;
\mlabel{it:sca1}
\item
Let $R$ be a unitary algebra and let $\bfk$ be a field. If $P$ is not a scalar product and is an \wmrbo, then its weight is unique.     \mlabel{it:sca2}
\end{enumerate}
\mlabel{pp:scalar}
\end{prop}
\begin{proof}
\meqref{it:sca1}
Let $P=\mu\id$. Then the \wmrbi in Eq.~\meqref{eq:grb}
means
$$\mu^2xy=\mu^2xy+\mu^2 xy +\mu \lambda xy +\kappa xy,$$
which simplifies to
$ \mu^2+\mu \lambda+\kappa=0.$
This holds precisely when $\kappa=-\mu^2-\mu\lambda$, for arbitrary $ \lambda\in \bfk.$
\smallskip

\noindent
\meqref{it:sca2}
Suppose that $P$ is not a scalar multiplication and is an \wmrbo. If $P$ has two weights $(\lambda_1,\kappa_1)$ and $(\lambda_2,\kappa_2)$, then for all $x,y\in R$,
$$ P(x)P(y)=P(xP(y)+P(x)y+\lambda_1 xy)+\kappa_1 xy, \quad P(x)P(y)=P(xP(y)+P(x)y+\lambda_2 xy)+\kappa_2 xy.$$
This means
\vsb
$$(\lambda_1-\lambda_2) P(xy)=(\kappa_2-\kappa_1 ) xy,\quad x,y\in R,$$
Taking $y=1$ gives $(\lambda_1-\lambda_2) P(x)=(\kappa_2-\kappa_1)x$ for all $x\in R$. If $\lambda_1\neq \lambda_2$, then $P(x)=\frac{\kappa_2-\kappa_1}{\lambda_1-\lambda_2}x$, contradicting the assumption that $P$ is not a scalar product. So $\lambda_1=\lambda_2$. Then $(\kappa_1-\kappa_2)x=0$, yielding $\kappa_1=\kappa_2$. Thus $(\lambda_1,\kappa_1)=(\lambda_2,\kappa_2)$.
\end{proof}

The following results generalize the interconnections among Rota-Baxter operators.
\begin{prop}
Let $P$ be a linear operator on an algebra $R$.
\begin{enumerate}
\item
Let $P$ be an \wmrbo of weight $(\lambda,\kappa)$ and let $\mu$ be a scalar in $\bfk$. Then $\mu P$ is an \wmrbo of weight $(\lambda \mu,\kappa\mu^2)$. The converse holds if $\mu$ is invertible in $\bfk$;
\mlabel{it:lambdap}
\item
Let $\tilde{P}:=-\lambda \id-P$, called the {\bf adjoint operator} of $P$. Then $P$ is an \wmrbo of weight $(\lambda,\kappa)$ if and only if $\tilde{P}$ is one.
\mlabel{it:tildep}
\end{enumerate}
\mlabel{prop:exte}
\end{prop}
\begin{proof}
(\mref{it:lambdap}) is obtained by multiplying $\mu^2$ across the \wmrbi \meqref{eq:grb}.
\smallskip
	
\noindent
(\mref{it:tildep})
If $P$ is an \wmrbo of weight $(\lambda,\kappa)$ on $R$, then for all $x,y\ \in R$,
\vsb
\begin{align*}
\tilde{P}(x)\tilde{P}(y)={}& (-\lambda \id-P)(x)(-\lambda \id-P)(y)\\
={}& \lambda^2xy+\lambda xP(y)+\lambda P(x)y+P(x)P(y)\\
		={}& \lambda^2xy+\lambda xP(y)+\lambda P(x)y+P\big(P(x)y\big)+P\big(xP(y)\big)+\lambda P(xy)+\kappa xy.
	\end{align*}
On the other hand,
\vsb
	\begin{align*}
		&\tilde{P}\Big(\tilde{P}(x)y\Big)+\tilde{P}\Big(x\tilde{P}(y)\Big)+\lambda \tilde{P}(xy)+\kappa xy\\
		={}& (-\lambda \id-P)\Big((-\lambda \id-P)(x)y\Big)+(-\lambda \id-P)\Big(x(-\lambda \id-P)(y)\Big)+\lambda(-\lambda \id-P)(xy)+\kappa xy\\
		={}& \lambda^2xy+\lambda P(x)y+\lambda P(xy)+P\big(P(x)y\big)+\lambda^2xy+\lambda xP(y)+\lambda P(xy)+P\big(xP(y)\big)\\
		&-\lambda^2xy-\lambda P(xy)+\kappa xy\\
		={}& \lambda^2xy+\lambda P(x)y+\lambda xP(y)+P\big(P(x)y\big)+P\big(xP(y)\big)+\lambda P(xy)+\kappa xy,
	\end{align*}
which agrees with $\tilde{P}(x)\tilde{P}(y)$. Hence $\tilde{P}$ is an \wmrbo of weight $(\lambda,\kappa)$.

The converse holds since $\tilde{\tilde{P}}=-\lambda\id - \tilde{P}=P$.
\end{proof}


The following property shows that \wmrbos provide a unified context for Rota-Baxter operators and modified Rota-Baxter operators with linear closedness.

\begin{theorem}
Let $P_1$ and $P_2$ be \wmrbos of weight $(\lambda_1,\kappa_1)$ and $(\lambda_2,\kappa_2)$ respectively. Suppose that $P_1$ and $P_2$ satisfy the compatible properties that, for $1\leq i\neq j\leq 2$, 
\vsb
\begin{equation}
P_i(u)P_j(v)=P_i(uP_j(v))+P_j(P_i(u)v)+\alpha_{ij} P_i(uv)+\beta_{ij} P_j(uv)+\gamma_{ij}uv, \,\quad u,v\in R,
\mlabel{eq:comij}
\end{equation}
for some $\alpha_{ij}, \beta_{ij}, \gamma_{ij}\in\bfk$. If $\beta_{12}+\alpha_{21}=\lambda_1$ and $\alpha_{12}+\beta_{21}=\lambda_2$,
then for all $a,b\in \bfk$, $Q:=Q_{a,b}:=aP_1+bP_2$ is an \wmrbo of weight $(a\lambda_1+b\lambda_2,a^2\kappa_1+b^2\kappa_2+ab(\gamma_{12}+\gamma_{21}))$.
	\mlabel{thm:genfij}
\end{theorem}
\begin{proof}
	For all $u,v\in R$, we have
	\begin{eqnarray*}
		Q(u)Q(v)
		&=&a^2P_1(u)P_1(v)+abP_1(u)P_2(v)+abP_2(u)P_1(v)+b^2P_2(u)P_2(v)\\
		&=&a^2P_1(uP_1(v)+P_1(u)v+\lambda_1 uv)+a^2\kappa_1 uv+abP_1(u)P_2(v)+abP_2(u)P_1(v)\\
		&&+b^2P_2(uP_2(v)+P_2(u)v+\lambda_2 uv)+b^2\kappa_2 uv\\
		&=&a^2P_1(uP_1(v))+a^2P_1(P_1(u)v)+a^2\lambda_1 P_1(uv)+a^2\kappa_1 uv+abP_1(uP_2(v))\\
		&&+abP_2(P_1(u)v)+ab\alpha_{12} P_1(uv)+ab\beta_{12}P_2(uv)+ab\gamma_{12}uv\\
		&&+abP_2(uP_1(v))+abP_1(P_2(u)v)+ab\alpha_{21} P_2(uv)+ab \beta_{21} P_1(uv)+ab\gamma_{21}uv\\
		&&+b^2P_2(uP_2(v)+b^2P_2(P_2(u)v)+b^2\lambda_2 P_2(uv)+b^2\kappa_2 uv\quad(\text{by Eq.~\meqref{eq:comij}})\\
&=&a^2P_1(uP_1(v))+a^2P_1(P_1(u)v)+abP_1(uP_2(v))+abP_2(P_1(u)v)\\
		&&+abP_2(uP_1(v))+abP_1(P_2(u)v)+b^2P_2(uP_2(v))+b^2P_2(P_2(u)v)\\
&&+(a^2\lambda_1 +ab\lambda_2) P_1(uv)+(ab\lambda_1+b^2\lambda_2) P_2(uv)\\
&&+(a^2\kappa_1  +ab\gamma_{12} +ab\gamma_{21} +b^2\kappa_2) uv.  \quad(\text{by $\alpha_{12}+\beta_{21}=\lambda_2$ and $\beta_{12}+\alpha_{21}=\lambda_1$})
	\end{eqnarray*}
On the other hand,
\vsb
\begin{eqnarray*}
		&&Q(uQ(v))+Q(Q(u)v)+(a\lambda_1+b\lambda_2)Q(uv)+ \big(a^2\kappa_1+b^2\kappa_2+ab(\gamma_{12}+\gamma_{21})\big) uv\\
		&=&(aP_1+bP_2)(auP_1(v)+buP_2(v))+(aP_1+bP_2)(aP_1(u)v+bP_2(u)v)\\
		&&+(a\lambda_1+b\lambda_2)(aP_1(uv)+bP_2(uv))+\big(a^2\kappa_1+b^2\kappa_2+ab(\gamma_{12}+\gamma_{21})\big) uv\\
		&=&a^2P_1(uP_1(v))+abP_1(uP_2(v))+abP_2(uP_1(v))+b^2P_2(uP_2(v))\\
		&&+a^2P_1(P_1(u)v)+abP_1(P_2(u)v)+abP_2(P_1(u)v)+b^2P_2(P_2(u)v)\\
		&&+(a\lambda_1+b\lambda_2)(aP_1(uv)+bP_2(uv))+\big(a^2\kappa_1+b^2\kappa_2+ab(\gamma_{12}+\gamma_{21})\big) uv.
	\end{eqnarray*}
This agrees with $Q(u)Q(v)$, proving the theorem.
\end{proof}

\vsb

We give some special cases. First taking $\kappa_i=0$ and $\beta_{ij}=\gamma_{ij}=0$ gives the matching Rota-Baxter operators introduced in~\mcite{ZyGG} motivated by multi-pre-Lie algebras from regularity structures~\mcite{BHZ}. Next taking $P_2=\tilde{P}_1$, we obtain

\begin{coro}
 	Let $P$ be an \wmrbo of weight $(\lambda,\kappa)$ on $R$.
 	\begin{enumerate}
 		\item
 		Let $\tilde{P}=-\lambda \id -P$. Then $P$ and $\tilde{P}$ satisfy the following compatible properties$:$
 		\begin{eqnarray*}
 			&&P(u)\tilde{P}(v)=P(u\tilde{P}(v))+\tilde{P}(P(u)v)-\kappa uv\\
 			&& \tilde{P}(u)P(v)=\tilde{P}(uP(v))+P(\tilde{P}(u)v)-\kappa uv,\,\quad u,v\in R.
 			\vse
 		\end{eqnarray*}
 		\mlabel{it:compro}
 		\vse
 		\item
 		For all $a,b\in \bfk$, the operator $Q:=Q_{a,b}=aP+b\tilde{P}$ is an \wmrbo of weight $(\lambda(a+b),ab\lambda^2+(a-b)^2\kappa)$.
 		\mlabel{it:combination}
 	\end{enumerate}
 	\mlabel{coro:genf2}
 \end{coro}
 \vsb
 \begin{proof}
 	(\mref{it:compro})
 	For any $u,v\in R$, we have
 	$$		P(u)\tilde{P}(v)=P(u)(-\lambda v-P(v))
 	=-\lambda P(u)v-P(u)P(v).
 	$$
 	On the other hand,
 	\begin{eqnarray*}
 		P(u\tilde{P}(v))+\tilde{P}(P(u)v)-\kappa uv
 		&=&P(-\lambda uv-uP(v))-\lambda P(u)v-P(P(u)v)-\kappa uv\\
 		&=&-\lambda P(uv)-P(uP(v))-P(P(u)v)-\kappa uv-\lambda P(u)v\\
 		&=&-P(u)P(v)-\lambda P(u)v,
 	\end{eqnarray*}
 	proving the first equation. The same argument proves the second compatibility condition.
 	\smallskip
 	
\noindent
(\mref{it:combination})
 	Taking $P_1=P$ and $P_2=\tilde{P_1}=\tilde{P}$ in Theorem~\mref{thm:genfij}, we have $\lambda_1=\lambda_2=\lambda$, and $\kappa_1=\kappa_2=\kappa$ by Proposition~\mref{prop:exte}~\meqref{it:tildep}. Taking $\alpha_{12}=\beta_{12}=0,\alpha_{21}=\beta_{21}=\lambda$, and $\gamma_{12}=-\kappa,\gamma_{21}=\lambda^2-\kappa$, then Eq.~\meqref{eq:comij} follows from Item~ (\mref{it:compro}). Thus
 	$\alpha_{12}+\beta_{21}=\lambda=\beta_{12}+\alpha_{21}.$
 	Then by Theorem~\mref{thm:genfij}, $Q$ is an \wmrbo of weight $(\lambda(a+b),ab\lambda^2+(a-b)^2\kappa)$.
 \end{proof}
 \vsb

Specialized to Rota-Baxter operators, Corollary~\mref{coro:genf2} gives

\begin{coro}
	Let $P$ be a Rota-Baxter operator of weight $\lambda$ on $R$ and let $\tilde{P}=-\lambda \id -P$.
	\begin{enumerate}
		\item
		For $a,b\in \bfk$, the operator $Q:=Q_{a,b}=aP+b\tilde{P}$ is an \wmrbo of weight $(\lambda(a+b),ab\lambda^2)$.
		\mlabel{it:genf1}
		\item
		For $\mu, \kappa\in \bfk$, let $a, b\in \bfk$ be the roots of $t^2-\mu t+\kappa$. Then $Q:=aP+b\tilde{P}$ is an \wmrbo of weight $(\mu\lambda,\kappa\lambda^2)$. In particular, if $P$ is a Rota-Baxter operator of weight 1, then $Q$ is an \wmrbo of weight $(\mu, \kappa)$.
		\mlabel{it:genf2}
	\end{enumerate}
	\mlabel{co:genf}
\end{coro}
\vsb
\begin{proof}
\meqref{it:genf1} follows from taking $\kappa=0$ in Corollary~\mref{coro:genf2}~\meqref{it:combination}.
\smallskip
	
	\noindent
	\meqref{it:genf2} follows from the fact that the roots $a, b$ of $t^2-\mu t+\kappa$ are determined by $t^2-(a+b)t+ab=0$.
\end{proof}

We now give some applications of the general results that provide a large class of examples of \wmrbos.
The following \wmrbo can be viewed as a variation of the $q$-integral operator (or Jackson integral operator)
\vsb
$$I_q(f)(x):=\sum_{k=1}^\infty f(q^kx)$$
\vsa
on the polynomial algebra $\bfk[x]$. See, for example, ~\mcite{Gub,Ro2}.

\begin{exam}\mlabel{exam:polynomial}Suppose that  $q$ is not a root of unity in $\bfk$.
 Define linear operators $P_+$  and $P_-$ on $\bfk[x]$ by
 \vsb
$$P_+(x^n)=\frac{x^n}{1-q^n}\quad\text{and}\quad P_-(x^n)=\frac{-q^nx^n}{1-q^n},\,n\geq 0.$$
By \cite[Example 1.1.8]{Gub},  $P_+$ and $P_-$ are Rota-Baxter operators of weight $-1$. Also $P_-=\id-P_+=\tilde{P}$. Let $Q:=aP_++bP_-$. Then
 \vsb
 $$Q(x^n) = \frac{(a-q^nb)x^n}{1-q^n}, \quad n\geq 0.$$
By Corollary~\mref{co:genf},  $Q$ is an \wmrbo of weight $(-(a+b),ab)$.
\end{exam}

\vsb
\begin{exam} On the algebra of Laurent series $R=\CC[t^{-1},t]]$, define the linear operators
\vsa
$$P_+\Big(\sum_{n=k}^\infty a_nt^n\Big)=\sum_{n=0}^\infty a_nt^n\quad\text{and}\quad P_-\Big(\sum_{n=k}^\infty a_nt^n\Big)=\sum_{n=k}^{-1} a_nt^n.$$
Then $P_+$ and $P_-$ are Rota-Baxter operators of weight $-1$, playing a critical role in renormalization of quantum field theory\,\mcite{CK00}. Further, $P_-=\id-P=\tilde{P}$.
Let $Q:=aP_++bP_-$. Then
\vsa
$$Q\Big(\sum_{n=k}^\infty a_nt^n\Big)=a\sum_{n=0}^\infty a_nt^n+b\sum_{n=k}^{-1}a_nt^n.$$
By Corollary~\mref{co:genf}, $Q$ is an \wmrbo of weight $(-(a+b),ab)$ on $\CC[t^{-1},t]]$.
\end{exam}

The following example of Rota-Baxter operators arises from the study of umbral calculus and Stirling numbers. See \mcite{Gum,Gst} for example.	
\begin{exam}
Consider the (divided) falling factorials
$$ t_k:= \frac{x(x-1)\cdots (x-k+1)}{k!}\in \bfk[x],\  k\geq 0.$$
Then the linear operator $P$ on $\bfk[x]$ defined by
$$ P: \bfk[x]\to \bfk[x], \quad t_k\mapsto t_{k+1},\ k\geq 0,$$
is a Rota-Baxter operator of weight $1$. Further
$\tilde{P}=-\id -P$ is given by
$$ \tilde{P}(t_k)=-t_k-t_{k+1}=-t_k\frac{x+1}{k+1}.$$
For $a, b\in \bfk$, $Q_{a,b}:=aP+b\tilde{P}$ is given by
\begin{equation} Q_{a,b}(t_k):=(aP+b\tilde{P})(t_k)=t_k\frac{(a-b)x-ak-b}{k+1}.
\mlabel{eq:qabop}
\end{equation}
By Corollary~\mref{co:genf}, $Q$ is an \wmrbo of weight $(a+b,ab)$. In particular, for given $\mu, \kappa\in \bfk$, let $a, b$ be the roots of $x^2-\mu x+\kappa$. So $a,b=\frac{\mu\pm\sqrt{\mu^2-4\kappa}}{2}$. Then $Q_{a,b}$ is an \wmrbo of weight $(\mu,\kappa)$.
\end{exam}

 \vsd
\section{Free commutative \wmrbs on a commutative algebra}
\mlabel{sec:FreeCommGRBA}

This section constructs the free commutative \wmrb on a commutative algebra by a generalized quasi-shuffle product or mixable shuffle product. We also provide a combinatorial interpretation in terms of colored Delannoy paths.

\vsb
\subsection{The construction of free commutative \wmrbs}
\mlabel{subsec:FreeCommGRBA}
We first define the definition.
\vsb
\begin{defn}
Let $A$ be a commutative algebra and let $\lambda,\kappa\in \bfk$. A {\bf free commutative \wmrb of weight $(\lambda,\kappa)$ on $A$} is a commutative \wmrb  $F_{e}(A)$ of weight $(\lambda,\kappa)$ together with an algebra homomorphism $j_{A}:A \to F_{e}(A)$ satisfying the following universal property: for any commutative \wmrb $(R,P)$ of weight $(\lambda,\kappa)$ and any algebra homomorphism $f:A\to R$, there is a unique \wmrb homomorphism $\free{f}:F_{e}(A)\to R$ such that $f=\free{f}\circ j_A$, that is, the following diagram
\vsb
$$ \xymatrix{ A \ar[rr]^{j_A}\ar[drr]^{f} && F_e(A) \ar[d]_{\free{f}} \\
	&& R}
\vsc
	$$
commutes.
\mlabel{de:free}
\end{defn}

For a given commutative $\bfk$-algebra $A$ with identity$1_A$ (or simply $\bfone$),
we next construct the free \wmrb on  $A$.
Consider the module
\begin{equation*}
\sha_e(A):=A\oplus A^{\ot 2}\oplus\cdots=\bigoplus_{n\geq 1}A^{\ot n}
\end{equation*}
spanned by the tensor powers of $A$. We first  define a $\bfk$-linear operator $\pg$ on $\sha_e(A)$ to be the right-shift operator:
\vsb
\begin{equation}
\pg:\sha_e(A)\to \sha_e(A)\ , \ \fraka\mapsto1_A\ot \fraka\in A^{\ot (n+1)}, \quad \quad \fraka\in A^{\ot n}.
\end{equation}
We next define a product $\dg$ on $\sha^+_e(A)$. For this purpose, we only need to define the product of two pure tensors and then extend the product by bilinearity.

For $\fraka\in A^{\ot (m+1)}$ where $m\geq 1$, denote $\fraka=a_0\ot \fraka'$ with $\fraka'\in A^{\ot m}$. Let $\fraka=a_0\ot \fraka'\in  A^{\ot m+1}$ and $\frakb=b_0\ot \frakb'\in  A^{\ot n+1} $ with $m,n\geq 0$, where $\fraka'=a_1\ot \cdots \ot a_m\in A^{\ot m}$ if $m\geq 1$ and $\frakb'=b_1\ot \cdots \ot b_n\in A^{\ot n}$ with $n\geq 1$. The product $\fraka \dg \frakb$  is  defined by a recursion.
\begin{equation}
\fraka\dg\frakb = \left \{\begin{array}{lll} a_0b_0, & m=n=0, \\
		a_0b_0\ot\frakb', & m=0, n\geq 1, \\
		a_0b_0\ot\fraka', & m\geq 1, n=0,\\
		a_0b_0\ot\Big(\fraka'\dg(1_A\ot\frakb') + (1_A\ot\fraka')\dg\frakb' + \lambda (\fraka'\dg\frakb')\Big)& \\
		+\kappa a_0b_0(\fraka'\dg\frakb'), & m, n\geq 1.
	\end{array} \right .
\mlabel{eq:dfndg}
\end{equation}

\begin{exam}Let $\fraka=a_0\ot a_1\in A^{\ot 2}$ and $\frakb=b_0\ot b_1\ot b_2\in A^{\ot 3}$. Then
\begin{eqnarray*}
\fraka\dg\frakb
&=&a_0b_0\ot\bigg(a_1\dg(1_A\ot b_1\ot b_2)+(1_A\ot a_1)\dg (b_1\ot b_2)+\lambda\big( a_1\dg (b_1\ot b_2)\big)\bigg)\\
&&+\kappa a_0b_0\big(a_1\dg (b_1\ot b_2)\big)\\
&=&a_0b_0\ot\bigg(a_1\ot b_1\ot b_2+b_1\ot\Big(a_1\dg(1_A\ot b_2)+(1_A\ot a_1)\dg b_2+\lambda a_1b_2\Big)+\kappa b_1a_1b_2\\
&&+\lambda a_1b_1\ot b_2\bigg)+\kappa a_0b_0a_1b_1\ot b_2\\
&=&a_0b_0\ot\bigg(a_1\ot b_1\ot b_2+b_1\ot a_1\ot b_2+b_1\ot b_2\ot a_1+\lambda b_1\ot a_1b_2+\lambda a_1b_1\ot b_2+\kappa b_1a_1b_2\bigg)\\&&+\kappa a_0b_0a_1b_1\ot b_2.
\end{eqnarray*}
\end{exam}
Let
 \vsb
\begin{equation}\mlabel{eq:j_A}
	j_A: A\to\sha_e(A)\ ,\   a\mapsto a, \quad a\in A,
\end{equation}
be the natural embedding. Then we have
\begin{equation*}
j_A(ab)=ab=a\dg b=j_A(a)\dg j_A(b)\ ,\  \quad a,b \in A.
\end{equation*}
Hence $j_A$ is an algebra homomorphism.

\begin{theorem}
Let $A$ be a commutative algebra and $\lambda, \kappa\in \bfk$. Let $\sha_e(A), \pg, j_A $ and $\dg$ be defined as above.
\begin{enumerate}
\item
The triple $(\sha_e(A), \dg, \pg)$ is a commutative \wmrb of weight $(\lambda, \kappa)$;
\mlabel{it:cgrba}
\vsb
\item
The  quadruple $(\sha_e(A), \dg, \pg, j_A)$ is the free commutative \wmrb of weight $(\lambda, \kappa)$ on $A$.
\mlabel{it:fcgrba}
\end{enumerate}
\mlabel{thm:ab}
\end{theorem}
\vsd
\begin{proof}
(\mref{it:cgrba})
First we prove the commutativity of $\dg$. For pure tensors $\fraka=a_0\ot a_1\ot\cdots\ot a_m$ and $\frakb=b_0\ot b_1\ot\cdots\ot b_n \in \sha_e(A)$ with $m,n\geq 0$, we will verify the equality
\begin{equation}
\fraka\dg\frakb=\frakb\dg\fraka
\mlabel{eq:com}
\end{equation}
by induction on the sum $s:=m+n\geq 0$. When $s=0$, then $m=n=0$ and so $\fraka=a_0,\frakb=b_0$. By the definition of $\dg$, we obtain
$$\fraka\dg\frakb=a_0\dg b_0=a_0b_0=b_0a_0=b_0\dg a_0=\frakb\dg\fraka.$$

For a given $k\geq 0$, assume that Eq.~\meqref{eq:com} holds for $s=m+n\le k$. Consider the case when $s=m+n=k+1\geq 1$. If $m=0$ or $n=0$, we write $\fraka=a_0,\frakb=b_0\ot\frakb'$, or $\fraka=a_0\ot \fraka',\frakb=b_0$. Then $$\fraka\dg\frakb=a_0b_0\ot\frakb'=b_0a_0\ot \frakb'=\frakb\dg\fraka \text{ or }
\fraka\dg\frakb=a_0b_0\ot\fraka'=b_0a_0\ot \fraka'=\frakb\dg\fraka.$$
Hence the commutativity holds in this case.
If $m,n\geq 1$, write $\fraka=a_0\ot\fraka'$ with $\fraka'\in A^{\ot (m-1)}$ and $\frakb=b_0\ot\frakb'$ with $\frakb'\in A^{\ot (n-1)}$. Then applying the induction hypothesis, we obtain
\begin{align*}
\fraka\dg\frakb
=& a_0b_0\ot\big(\fraka'\dg(1_A\ot\frakb')\big)+a_0b_0\ot\big((1_A\ot\fraka')\dg\frakb'\big)+\lambda a_0b_0\ot(\fraka'\dg\frakb')+\kappa a_0b_0(\fraka'\dg\frakb')\\
=& b_0a_0\ot\big((1_A\ot\frakb')\dg\fraka'\big)+b_0a_0\ot\big(\frakb'\dg(1_A\ot\fraka')\big)+\lambda b_0a_0\ot(\frakb'\dg\fraka')+\kappa b_0a_0(\frakb'\dg\fraka')\\
=& \frakb\dg\fraka.
\end{align*}
This completes the inductive proof of the commutativity of $\dg$.

Next we verify the associativity of $\dg$. For this we only need to check that, for all pure tensors $\fraka, \frakb, \frakc\in \sha_e(A)$,  the following associativity holds.
\begin{equation}
(\fraka\dg\frakb)\dg\frakc=\fraka\dg(\frakb\dg\frakc).
\mlabel{eq:assdg}
\end{equation}
Let $\fraka=a_0\ot a_1\ot \cdots\ot a_m\in A^{\ot (m+1)},\frakb=b_0\ot b_1\ot \cdots\ot b_n\in A^{\ot (n+1)},\frakc=c_0\ot c_1\ot \cdots\ot c_r\in A^{\ot (r+1)}$ with $m,n,r\geq 0$. We will apply induction on the sum $s:=m+n+r\geq 0$ to verify Eq.~\meqref{eq:assdg}. When $m+n+r=0$, that is, $m=n=r=0$ and so $\fraka=a_0,\frakb=b_0,\frakc=c_0$. Then Eq.~\meqref{eq:assdg} follows from Eq.~(\mref{eq:dfndg}) and the associativity of the algebra $A$:
\begin{equation*}
(\fraka\dg\frakb)\dg\frakc=(a_0b_0)\dg c_0=(a_0b_0)c_0=a_0(b_0c_0)=a_0\dg(b_0c_0)=\fraka\dg(\frakb\dg\frakc).
\end{equation*}

For a fixed $k\geq 0$,  assume that Eq.~(\mref{eq:assdg}) holds for $m+n+r\le k$ and consider $m+n+r=k+1\geq 1$. So at most two of $m,n,r$ are zero. We consider the three cases of exactly two, one or none of $m, n, r$ being zero.

{\bf Case 1: Two of $m,n,r$ are $0$}. Without loss of generality, we might suppose $m=n=0$. Then $\fraka=a_0, \frakb=b_0, \frakc=c_0\ot\frakc'$ with $\frakc'\in A^{\ot r}, r\geq 1$. By  Eq.~(\mref{eq:dfndg}), we have
\begin{align*}
(\fraka\dg\frakb)\dg\frakc
={}& (a_0b_0)\dg(c_0\ot\frakc')\\
={}& (a_0b_0)c_0\ot\frakc'\\
={}& a_0(b_0c_0)\ot\frakc'\\
={}& a_0\dg(b_0c_0\ot\frakc')\\
={}& a_0\dg\big(b_0\dg(c_0\ot\frakc')\big)\\
={}& \fraka\dg(\frakb\dg\frakc),
\end{align*}
as desired.

{\bf Case 2: Exactly one of $m,n,r$ is $0$}. Without loss of generality, we might suppose $m=0$. Then $\fraka=a_0, \frakb=b_0\ot\frakb'$ with $\frakb'\in A^{\ot n}, n\geq 1$,  and $\frakc=c_0\ot\frakc'$ with $\frakc'\in A^{\ot r}, r\geq 1$. By Eq.~(\mref{eq:dfndg}), we also have the desired associativity:
\begin{align*}
(\fraka\dg\frakb)\dg\frakc
={}& \big(a_0\dg (b_0\ot\frakb')\big)\dg(c_0\ot\frakc')\\
={}& (a_0b_0\ot\frakb')\dg(c_0\ot\frakc')\\
={}& (a_0b_0)c_0\ot\big(\frakb'\dg(1_A\ot\frakc')\big)+(a_0b_0)c_0\ot\big((1_A\ot\frakb')\dg \frakc'\big)\\
&+\lambda(a_0b_0)c_0\ot(\frakb'\dg \frakc')+\kappa(a_0b_0)c_0(\frakb'\dg \frakc')\\
={}& a_0(b_0c_0)\ot\big(\frakb'\dg(1_A\ot\frakc')\big)+a_0(b_0c_0)\ot\big((1_A\ot\frakb')\dg \frakc'\big)\\
&+\lambda a_0(b_0c_0)\ot(\frakb'\dg \frakc')+\kappa a_0(b_0c_0)(\frakb'\dg \frakc')\\
={}& a_0\dg\Big(b_0c_0\ot\big(\frakb'\dg(1_A\ot\frakc')\big)+b_0c_0\ot\big((1_A\ot\frakb')\dg \frakc'\big)\\
&+\lambda b_0c_0\ot(\frakb'\dg \frakc')+\kappa b_0c_0(\frakb'\dg \frakc')\Big)\\
={}& \fraka \dg(\frakb \dg \frakc).
\end{align*}

{\bf Case 3: None of $m,n,r$ is zero}. We denote $\fraka=a_0\ot\fraka',\frakb=b_0\ot\frakb',\frakc=c_0\ot\frakc'$ with $\fraka'\in A^{\ot m},\frakb'\in A^{\ot n},\frakc'\in A^{\ot r}$. Then by  Eq.~(\mref{eq:dfndg}), we obtain
\begin{align*}
&(\fraka \dg \frakb)\dg \frakc\\
={}& \Big((a_0\ot\fraka')\dg(b_0\ot\frakb')\Big)\dg(c_0\ot\frakc')\\
={}& \Big(a_0b_0\ot\big(\fraka'\dg(1_A\ot\frakb')\big)
+a_0b_0\ot\big((1_A\ot\fraka')\dg \frakb'\big)\\
&+\lambda a_0b_0\ot(\fraka'\dg \frakb')+\kappa a_0b_0(\fraka'\dg \frakb')\Big)\dg(c_0\ot\frakc')\\
={}& \Big(a_0b_0\ot\big(\fraka'\dg(1_A\ot\frakb')\big)\Big)\dg(c_0\ot\frakc')+\Big(a_0b_0\ot\big((1_A\ot\fraka')\dg \frakb'\big)\Big)\dg(c_0\ot\frakc')\\
&+\lambda \Big(a_0b_0\ot(\fraka'\dg \frakb')\Big)\dg(c_0\ot\frakc')+\kappa\Big(a_0b_0(\fraka'\dg \frakb')\Big)\dg(c_0\ot\frakc')\\
={}& a_0b_0c_0\ot\Big(\underline{\big(\fraka'\dg(1_A\ot\frakb')\big)\dg(1_A\ot\frakc')}\Big)+a_0b_0c_0\ot\bigg(\Big(1_A\ot\big(\fraka'\dg(1_A\ot\frakb')\big)\Big)\dg \frakc'\bigg)\\
&+\lambda a_0b_0c_0\ot\Big(\big(\fraka'\dg(1_A\ot\frakb')\big)\dg \frakc'\Big)+\kappa a_0b_0c_0\Big(\big(\fraka'\dg(1_A\ot\frakb')\big)\dg \frakc'\Big)\\
&+a_0b_0c_0\ot\Big(\big((1_A\ot\fraka')\dg \frakb'\big)\dg(1_A\ot\frakc')\Big)+a_0b_0c_0\ot\bigg(\Big(1_A\ot\big((1_A\ot\fraka')\dg \frakb'\big)\Big)\dg \frakc'\bigg)\\
&+\lambda a_0b_0c_0\ot\Big(\big((1_A\ot\fraka')\dg \frakb'\big)\dg \frakc'\Big)+\kappa a_0b_0c_0\Big(\big((1_A\ot\fraka')\dg \frakb'\big)\dg \frakc'\Big)\\
&+\lambda a_0b_0c_0\ot\Big((\fraka' \dg \frakb')\dg(1_A\ot\frakc')\Big)+\lambda a_0b_0c_0\ot\Big(\big(1_A\ot(\fraka' \dg \frakb')\big)\dg \frakc'\Big)\\
&+\lambda^2 a_0b_0c_0\ot\Big((\fraka'\dg \frakb')\dg \frakc'\Big)+\lambda\kappa a_0b_0c_0\Big((\fraka'\dg \frakb')\dg \frakc'\Big)\\
&+\kappa a_0b_0c_0\Big((\fraka'\dg \frakb')\dg (1_A\ot\frakc')\Big).
\end{align*}
Applying the induction hypothesis to the underlined factor and then using Eq.~\meqref{eq:dfndg}, we derive
\begin{align*}
&(\fraka \dg \frakb)\dg \frakc\\
={}& a_0b_0c_0\ot\bigg(\fraka'\dg\Big(1_A\ot\big(\frakb'\dg(1_A\ot\frakc')\big)\Big)\bigg)+a_0b_0c_0\ot\bigg(\fraka'\dg\Big(1_A\ot\big((1_A\ot\frakb')\dg \frakc'\big)\Big)\bigg)\\
&+\lambda a_0b_0c_0\ot\Big(\fraka'\dg\big(1_A\ot(\frakb'\dg \frakc')\big)\Big)+\kappa a_0b_0c_0\ot\Big(\fraka'\dg(\frakb'\dg\frakc')\Big)\\
&+a_0b_0c_0\ot\bigg(\Big(1_A\ot\big(\fraka'\dg(1_A\ot\frakb')\big)\Big)\dg \frakc'\bigg)+\lambda a_0b_0c_0\ot\Big(\big(\fraka'\dg(1_A\ot\frakb')\big)\dg \frakc'\Big)\\
&+\kappa a_0b_0c_0\Big(\big(\fraka'\dg(1_A\ot\frakb')\big)\dg \frakc'\Big)+a_0b_0c_0\ot\Big(\big((1_A\ot\fraka')\dg \frakb'\big)\dg(1_A\ot\frakc')\Big)\\
&+a_0b_0c_0\ot\bigg(\Big(1_A\ot\big((1_A\ot\fraka')\dg \frakb'\big)\Big)\dg \frakc'\bigg)+\lambda a_0b_0c_0\ot\Big(\big((1_A\ot\fraka')\dg \frakb'\big)\dg \frakc'\Big)\\
&+\kappa a_0b_0c_0\Big(\big((1_A\ot\fraka')\dg \frakb'\big)\dg \frakc'\Big)+\lambda a_0b_0c_0\ot\Big((\fraka' \dg \frakb')\dg(1_A\ot\frakc')\Big)\\
&+\lambda a_0b_0c_0\ot\Big(\big(1_A\ot(\fraka' \dg \frakb')\big)\dg \frakc'\Big)+\lambda^2 a_0b_0c_0\ot\Big((\fraka'\dg \frakb')\dg \frakc'\Big)\\
&+\lambda\kappa a_0b_0c_0\Big((\fraka'\dg \frakb')\dg \frakc'\Big)+\kappa a_0b_0c_0\Big((\fraka'\dg \frakb')\dg (1_A\ot\frakc')\Big).
\vsb
\end{align*}

On the other hand, we have
\vsb
\begin{align*}
&\fraka\dg(\frakb\dg\frakc)\\
={}& (a_0\ot\fraka')\dg\Big((b_0\ot\frakb')\dg(c_0\ot\frakc')\Big)\\
={}& (a_0\ot\fraka')\dg\Big(b_0c_0\ot\big(\frakb'\dg(1_A\ot\frakc')\big)\Big)+(a_0\ot\fraka')\dg\Big(b_0c_0\ot\big((1_A\ot\frakb')\dg \frakc'\big)\Big)\\
&+\lambda(a_0\ot\fraka')\dg\big(b_0c_0\ot (\frakb'\dg \frakc')\big)+\kappa(a_0\ot\fraka')\dg\big(b_0c_0(\frakb'\dg \frakc')\big)\\
={}& a_0b_0c_0\ot\bigg(\fraka'\dg\Big(1_A\ot\big(\frakb'\dg(1_A\ot\frakc')\big)\Big)\bigg)+a_0b_0c_0\ot\Big((1_A\ot\fraka')\dg\big(\frakb'\dg(1_A\ot\frakc')\big)\Big)\\
&+\lambda a_0b_0c_0\ot\Big(\fraka'\dg\big(\frakb'\dg(1_A\ot\frakc')\big)\Big)+\kappa a_0b_0c_0\Big(\fraka'\dg\big(\frakb'\dg(1_A\ot\frakc')\big)\Big)\\
&+a_0b_0c_0\ot\bigg(\fraka'\dg\Big(1_A\ot\big((1_A\ot\frakb')\dg \frakc'\big)\Big)\bigg)+a_0b_0c_0\ot\Big(\underline{(1_A\ot\fraka')\dg\big((1_A\ot\frakb')\dg \frakc'\big)}\Big)\\
&+\lambda a_0b_0c_0\ot\Big(\fraka'\dg\big((1_A\ot\frakb')\dg \frakc'\big)\Big)+\kappa a_0b_0c_0\Big(\fraka'\dg\big((1_A\ot\frakb')\dg \frakc'\big)\Big)\\
&+\lambda a_0b_0c_0\ot\Big(\fraka'\dg\big(1_A\ot(\frakb'\dg \frakc')\big)\Big)+\lambda a_0b_0c_0\ot\Big((1_A\ot\fraka')\dg (\frakb'\dg \frakc')\Big)\\
&+\lambda^2 a_0b_0c_0\ot\Big(\fraka'\dg (\frakb'\dg \frakc')\Big)+\lambda\kappa a_0b_0c_0\Big(\fraka'\dg(\frakb'\dg \frakc')\Big)\\
&+\kappa a_0b_0c_0\Big((1_A\ot\fraka')\dg(\frakb'\dg\frakc')\Big).
\end{align*}
 \vsb
Applying the induction hypothesis to the underlined factor and then using Eq.~\meqref{eq:dfndg}, we obtain
 \vsb
 \begin{align*}
&\fraka \dg (\frakb \dg \frakc)\\
={}& a_0b_0c_0\ot\bigg(\fraka'\dg\Big(1_A\ot\big(\frakb'\dg(1_A\ot\frakc')\big)\Big)\bigg)+a_0b_0c_0\ot\Big((1_A\ot\fraka')\dg\big(\frakb'\dg(1_A\ot\frakc')\big)\Big)\\
&+\lambda a_0b_0c_0\ot\Big(\fraka'\dg\big(\frakb'\dg(1_A\ot\frakc')\big)\Big)+\kappa a_0b_0c_0\Big(\fraka'\dg\big(\frakb'\dg(1_A\ot\frakc')\big)\Big)\\
&+a_0b_0c_0\ot\bigg(\fraka'\dg\Big(1_A\ot\big((1_A\ot\frakb')\dg \frakc'\big)\Big)\bigg)+a_0b_0c_0\ot\bigg(\Big(1_A\ot\big(\fraka'\dg(1_A\ot\frakb')\big)\Big)\dg \frakc'\bigg)\\
&+a_0b_0c_0\ot\bigg(\Big(1_A\ot\big((1_A\ot\fraka')\dg \frakb'\big)\Big)\dg \frakc'\bigg)+\lambda a_0b_0c_0\ot\Big(\big(1_A\ot(\fraka'\dg \frakb')\big)\dg \frakc'\Big)\\
&+\kappa a_0b_0c_0\ot\big((\fraka'\dg \frakb') \dg \frakc'\big)+\lambda a_0b_0c_0\ot\Big(\fraka'\dg\big((1_A\ot\frakb')\dg \frakc'\big)\Big)\\
&+\kappa a_0b_0c_0\Big(\fraka'\dg\big((1_A\ot\frakb')\dg \frakc'\big)\Big)+\lambda a_0b_0c_0\ot\Big(\fraka'\dg\big(1_A\ot(\frakb'\dg \frakc')\big)\Big)\\
&+\lambda a_0b_0c_0\ot\Big((1_A\ot\fraka')\dg (\frakb'\dg \frakc')\Big)+\lambda^2 a_0b_0c_0\ot\Big(\fraka'\dg (\frakb'\dg \frakc')\Big)\\
&+\lambda\kappa a_0b_0c_0\Big(\fraka'\dg(\frakb'\dg \frakc')\Big)+\kappa a_0b_0c_0\Big((1_A\ot\fraka')\dg(\frakb'\dg\frakc')\Big).
\end{align*}
Now by the induction hypothesis, the $i$-th term in the expansion of $(\fraka\dg\frakb)\dg\frakc$ matches with the $\sigma(i)$-th term in the expansion of $\fraka\dg(\frakb\dg\frakc)$, where the permutation $\sigma\in \sum_{16}$ is
\begin{equation*}
	\begin{pmatrix}
		i\\ \sigma(i)
	\end{pmatrix}
	=
	\left ( \begin{array}{cccccccccccccccc}
		1 & 2 & 3 & 4 & 5 & 6 & 7 & 8 & 9 & 10 & 11 & 12 & 13 & 14 & 15 & 16\\
		1 & 5 & 12 & 9 & 6 & 10 & 11 & 2 & 7 & 13 & 16 & 3 & 8 & 14 & 15 & 4
	\end{array} \right).
\end{equation*}
This completes the inductive proof of Eq.~(\mref{eq:assdg}) and hence the associativity of $\dg$.

In summary, we have proved that $(\sha_e(A), \dg)$ is a commutative algebra with unit $1_A$.

We finally prove that $\pg$ is an \wmrbo of weight $(\lambda, \kappa)$ on $\sha_e(A)$. Let $\fraka, \frakb\in\sha_e(A)$. By Eq.~\meqref{eq:dfndg}, we have
\vsb
\begin{align*}
\pg(\fraka)\dg \pg(\frakb)
={}& (1_A\ot \fraka)\dg(1_A\ot \frakb)\\
={}& 1_A\ot\Big(\fraka\dg(1_A\ot \frakb)\Big)+1_A\ot\Big((1_A\ot\fraka)\dg \frakb\Big)+\lambda1_A\ot(\fraka\dg\frakb)+\kappa\fraka\dg\frakb\\
={}& \pg\big(\fraka\dg \pg(\frakb)\big)+\pg\big(\pg(\fraka)\dg\frakb\big)+\lambda \pg(\fraka\dg\frakb)+\kappa\fraka\dg\frakb.
\end{align*}
Thus $\pg$ is an \wmrbo of weight $(\lambda, \kappa)$ on $\sha_e(A)$, and so $(\sha_e(A), \dg, \pg)$ is a commutative \wmrb of weight $(\lambda, \kappa)$.
\smallskip

\noindent
(\mref{it:fcgrba})
To verify the universal property of  $(\sha_e(A), \dg, \pg, j_A)$ as given in Definition~\mref{de:free}, let $(R, P)$ be a commutative \wmrb of weight $(\lambda, \kappa)$ and let $f:A\to R$ be an algebra homomorphism. For any pure tensors $\fraka=a_1\ot\cdots \ot a_m\in A^{\ot m}$ with $m\geq 1$, we use induction on $m\geq 1$ to define an \wmrb homomorphism $\bar f:\sha_e(A)\to R$. If $m=1$, we define $\bar f(a)=f(a)$. Assume that $\bar f(\fraka)$ has been defined for $m\leq k$ with any given $k\geq 1$. Consider $\fraka=a_1\ot\fraka'\in A^{\ot (k+1)}$ for $\fraka'\in A^{\ot k}$. Then
\vsb
\begin{equation}\mlabel{eq:adgp}
\fraka=a_1\ot \fraka'=a_1\dg(1_A\ot \fraka')=a_1\dg \pg(\fraka').
\end{equation}
Define
\vsb
\begin{equation}\mlabel{eq:dfnf}
\bar f(\fraka)=f(a_1)P\big(\bar f(\fraka')\big),
\end{equation}
where $\bar f(\fraka')$ is well-defined by the induction hypothesis.
From the definition of $\bar f$, we have $\bar{f}\circ j_A=f$.

Next we verify that $\bar f$ is an \wmrb homomorphism. By Eq.~\meqref{eq:dfnf}, we obtain
\vsb
\begin{equation*}
\bar f\big(\pg(\fraka)\big)=\bar f(1_A\ot \fraka)=f(1_A)P\big(\bar f(\fraka)\big)=1_R P\big(\bar f(\fraka)\big)=P\big(\bar f(\fraka)\big).
\end{equation*}
This gives
\vsb
\begin{equation}
\bar f\circ \pg=P\circ\bar f.
\mlabel{eq:operatorcomm}
\end{equation}
Then it suffices to prove that $\bar f$ satisfies
\begin{equation}
\bar f(\fraka\dg\frakb)=\bar f(\fraka)\bar f(\frakb),\,\ \quad \fraka\in A^{\ot m}, \frakb\in A^{\ot n}\,\, \text{for}\, m, n\geq 1.
\mlabel{eq:algebrahomo}
\end{equation}
We will carry out the proof by induction on $m+n\geq 2$. When $m+n=2$, then $\fraka, \frakb\in A$. By Eq.~\meqref{eq:dfnf}, we have $\bar f(\fraka\dg\frakb)=\bar f(\fraka \frakb)=f(\fraka\frakb)=f(\fraka)f(\frakb)=\bar f(\fraka)\bar f(\frakb)$. So Eq.~\meqref{eq:algebrahomo} is true for $m+n=2$. For a given $k\geq 2$, assume that Eq.~\meqref{eq:algebrahomo} holds for $m+n\le k$. Now consider $\fraka=a_1\ot \fraka'\in A^{\ot m}$ and $\frakb=b_1\ot \frakb'\in A^{\ot n}$ with $m+n=k+1\geq 3$. Then we get
\begin{align*}
&\bar f(\fraka\dg\frakb)\\
={}& \bar f\big((a_1\ot\fraka')\dg(b_1\ot\frakb')\big) \\
={}& \bar f\Big(\big(a_1\dg \pg(\fraka')\big)\dg\big(b_1\dg \pg(\frakb')\big)\Big) \quad(\text{by Eq.~\meqref{eq:adgp}})\\
={}& \bar f\Big(\big(a_1\dg b_1\big)\dg\big(\pg(\fraka')\dg \pg(\frakb')\big)\Big) \\
={}& \bar f\bigg((a_1b_1)\dg\Big(\pg\big(\pg(\fraka')\dg\frakb'\big)+\pg\big(\fraka'\dg \pg(\frakb')\big)+\lambda  \pg(\fraka'\dg\frakb')+\kappa \fraka'\dg\frakb'\Big)\bigg)\\
={}& \bar f\Big((a_1b_1)\dg \pg\big(\pg(\fraka')\dg\frakb'\big)\Big)+\bar f\Big((a_1b_1)\dg \pg\big(\fraka'\dg \pg(\frakb')\big)\Big)\\
&+\bar f\big((a_1b_1)\dg \lambda  \pg(\fraka'\dg\frakb')\big)+\bar f\big((a_1b_1)\dg (\kappa \fraka'\dg\frakb')\big) \\
={}& f(a_1b_1)P\Big(\bar f\big(\pg(\fraka')\dg\frakb'\big)\Big)+f(a_1b_1)P\Big(\bar f\big(\fraka'\dg \pg(\frakb')\big)\Big)\\
&+f(a_1b_1)P\big(\bar f(\lambda\fraka'\dg\frakb')\big)+\bar f\big((a_1b_1)\dg (\kappa \fraka'\dg\frakb')\big) \quad(\text{by Eq.~\meqref{eq:dfnf}})\\
={}& f(a_1b_1)P\Big(\bar f\big(\pg(\fraka')\big)\bar f(\frakb')\Big)+f(a_1b_1)P\Big(\bar f(\fraka')\bar f \big(\pg(\frakb')\big)\Big)\\
&+f(a_1b_1)P\big(\lambda\bar f(\fraka')\bar f(\frakb')\big)+\kappa\bar f(a_1b_1)\bar f(\fraka')\bar f(\frakb') \quad(\text{by the induction hypothesis})\\
={}& f(a_1b_1)\Big(P\big(P(\bar f(\fraka'))\bar f(\frakb')\big)+P\big(\bar f(\fraka')P(\bar f(\frakb'))\big)+\lambda P\big(\bar f(\fraka')\bar f(\frakb')\big)+\kappa \bar f(\fraka')\bar f(\frakb')\Big) \quad(\text{by Eq.~\meqref{eq:operatorcomm}})\\
={}& f(a_1)f(b_1)P\big(\bar f(\fraka')\big)P\big(\bar f(\frakb')\big)\\
={}& \Big(f(a_1)P\big(\bar f(\fraka')\big)\Big)\Big(f(b_1)P\big(\bar f(\frakb')\big)\Big)\\
={}& \bar f\big(a_1\dg \pg(\fraka')\big)\bar f\big(b_1\dg \pg(\frakb')\big) \quad(\text{by Eq.~\meqref{eq:dfnf}})\\
={}& \bar f(\fraka)\bar f(\frakb) \quad(\text{by Eq.~\meqref{eq:adgp}}).
\end{align*}
This completes the induction and thus the proof of Eq.~(\mref{eq:algebrahomo}). The uniqueness of $\bar{f}$ follows from the facts that $\bar{f}$ needs to be an \wmrb homomorphism and $\bar{f}\circ j_A=f$.
Thus the quadruple $(\sha_e(A), \dg, \pg, j_A)$ has the desired universal property.
\end{proof}
\vsd

\subsection{The special case of $A=\bfk$ }
\mlabel{subsec:spec}
In this subsection, we give an explicit description of the product in the free commutative \wmrb $\sha_e(A)$ over $A$ in the special case when $A=\bfk$.

In this case, we have
\vsb
$$\sha_e(\bfk)=\bigoplus_{n\geq 0}\bfk^{\ot (n+1)}=\bigoplus_{n\geq 0}\bfk \bfone^{\ot (n+1)}.$$
Furthermore, the product  $\dg$ in $\sha_e(\bfk)$ can be formulated explicitly as follow.
The formula includes as special cases the formula of free commutative Rota-Baxter algebras (see~\cite[Proposition~3.2.3]{Gub} which is recalled in Eq.\,\meqref{eq:mix}) and the formula of free commutative modified Rota-Baxter algebras (see~\cite[Proposition 3.4]{ZGG19}).

\begin{prop}
For $m,n\geq 0$ and $\lambda, \kappa \in \bfk$, we have \begin{equation}\mlabel{eq:powerzero}
\bfone^{\ot (m+1)} \dg \bfone^{\ot (n+1)}=\sum_{r=0}^{\min\{m,n\}}\sum_{i=0}^{r}\dbinom{m+n-r}{m}\dbinom{m}{r}\dbinom{r}{i}\lambda^{r-i}\kappa^i\bfone^{\ot (m+n+1-r-i)}.
\end{equation}
The product $\bfone^{\ot (m+1)} \dg \bfone^{\ot (n+1)}$ can also be expressed as a linear combination of the basis elements of $\sha_e(\bfk)$ as follows.
\vsb
\begin{eqnarray*} \bfone^{\ot (m+1)} \dg \bfone^{\ot (n+1)}&=&\sum_{\ell=0}^{2\min\{m,n\}} \sum_{i=0}^{\ell}\dbinom{m+n-\ell+i}{m}\dbinom{m}{\ell-i}\dbinom{\ell-i}{i}\lambda^{\ell-2i}\kappa^i \bfone^{\ot(m+n+1-\ell)}.
\end{eqnarray*}
\mlabel{pp:kprod}
\vsd
\end{prop}

\begin{proof}
We prove Eq.~(\mref{eq:powerzero}) by induction on $s:=m+n\geq 0$. When $s=0$, we have $m=n=0$, and so $\bfone \dk \bfone=\bfone$.
For a given $k\geq 0$, assume that Eq.~(\mref{eq:powerzero}) holds for $0\le s\le k$. Consider the case when $s=k+1$. If one of $m,n$ is zero, without loss of generality, we suppose that $m=0$. Then
\vsb
$$\bfone \dk \bfone^{\ot (n+1)}=\bfone^{\ot (n+1)}=\sum_{r=0}^{0}\sum_{i=0}^{r}\dbinom{n-r}{0}\dbinom{0}{r}\dbinom{r}{i}\lambda^{r-i}\kappa^i\bfone^{\ot (n+1-r-i)},$$
as needed. If $m,n\geq 1$, we can assume $m\le n$ because of the commutativity of $\dg$. Then by  the induction hypothesis, we obtain
\vsb
\begin{align}\notag
&\bfone^{\ot (m+1)} \dg \bfone^{\ot (n+1)}\\ \notag
={}& \bfone\ot\Big(\bfone^{\ot (m+1)} \dg \bfone^{\ot n}\Big)+\bfone\ot \Big(\bfone^{\ot m} \dg \bfone^{\ot (n+1)}\Big)+\lambda \bfone\ot \Big(\bfone^{\ot m} \dg \bfone^{\ot n}\Big)+\kappa \bfone^{\ot m} \dg \bfone^{\ot n} \\ \notag
={}& \sum_{r=0}^{\min\{m,n-1\}}\sum_{i=0}^{r}\dbinom{m+n-1-r}{m}\dbinom{m}{r}\dbinom{r}{i}\lambda^{r-i}\kappa^i\bfone^{\ot (m+n+1-r-i)}\\ \notag
&+\sum_{r=0}^{\min\{m-1,n\}}\sum_{i=0}^{r}\dbinom{m+n-1-r}{m-1}\dbinom{m-1}{r}\dbinom{r}{i}\lambda^{r-i}\kappa^i\bfone^{\ot (m+n+1-r-i)}\\ \notag
&+\sum_{r=0}^{\min\{m-1,n-1\}}\sum_{i=0}^{r}\dbinom{m+n-2-r}{m-1}\dbinom{m-1}{r}\dbinom{r}{i}\lambda^{r+1-i}\kappa^i\bfone^{\ot (m+n-r-i)}\\ \notag
&+\sum_{r=0}^{\min\{m-1,n-1\}}\sum_{i=0}^{r}\dbinom{m+n-2-r}{m-1}\dbinom{m-1}{r}\dbinom{r}{i}\lambda^{r-i}\kappa^{i+1}\bfone^{\ot (m+n-1-r-i)}\\ \notag
={}& \underline{\sum_{r=0}^{\min\{m,n-1\}}\sum_{i=0}^{r}\dbinom{m+n-1-r}{m}\dbinom{m}{r}\dbinom{r}{i}\lambda^{r-i}\kappa^i\bfone^{\ot (m+n+1-r-i)}}\\ \notag
&+\sum_{r=0}^{m-1}\sum_{i=0}^{r}\dbinom{m+n-1-r}{m-1}\dbinom{m-1}{r}\dbinom{r}{i}\lambda^{r-i}\kappa^i\bfone^{\ot (m+n+1-r-i)}\quad(\text{by $\min\{m-1,n\}=m-1$})\\ \notag
&+\sum_{r=0}^{m-1}\sum_{i=0}^{r}\dbinom{m+n-2-r}{m-1}\dbinom{m-1}{r}\dbinom{r}{i}\lambda^{r+1-i}\kappa^i\bfone^{\ot (m+n-r-i)}\quad(\text{by $\min\{m-1,n-1\}=m-1$})\\ \notag
&+\sum_{r=0}^{m-1}\sum_{i=0}^{r}\dbinom{m+n-2-r}{m-1}\dbinom{m-1}{r}\dbinom{r}{i}\lambda^{r-i}\kappa^{i+1}\bfone^{\ot (m+n-1-r-i)}\quad(\text{by $\min\{m-1,n-1\}=m-1$}).
\vsb
\end{align}

Now if $m=n$, then $\min\{m,n-1\}=n-1=m-1$; while if $m<n$, then $m\le n-1$ and so $\min\{m,n-1\}=m$. In either case, the underlined term in the above sum becomes
\vsb
\begin{equation}
\begin{split}
	&\sum_{r=0}^{\min\{m,n-1\}}\sum_{i=0}^{r}\dbinom{m+n-1-r}{m}\dbinom{m}{r}\dbinom{r}{i}\lambda^{r-i}\kappa^i\bfone^{\ot (m+n+1-r-i)} \\
	={}& \sum_{r=0}^{m-1}\sum_{i=0}^{r}\dbinom{m+n-1-r}{m}\dbinom{m}{r}\dbinom{r}{i}\lambda^{r-i}\kappa^i\bfone^{\ot (m+n+1-r-i)}\\
	={}& \sum_{r=0}^{m}\sum_{i=0}^{r}\dbinom{m+n-1-r}{m}\dbinom{m}{r}\dbinom{r}{i}\lambda^{r-i}\kappa^i\bfone^{\ot (m+n+1-r-i)}\\
	&\quad(\text{by the convention $\dbinom{m+n-1-m}{m}=\dbinom{n-1}{m}=\dbinom{m-1}{m}=0$}).
\end{split}
\mlabel{eq:m}
\vsb
\end{equation}

Applying this equality to the underlined term gives
\vsb
\begin{align*}
		&1^{\ot (m+1)} \dg 1^{\ot (n+1)}\\
		={}& \sum_{r=0}^{m}\sum_{i=0}^{r}\dbinom{m+n-1-r}{m}\dbinom{m}{r}\dbinom{r}{i}\lambda^{r-i}\kappa^i\bfone^{\ot (m+n+1-r-i)}\\
		&+\sum_{r=0}^{m-1}\sum_{i=0}^{r}\dbinom{m+n-1-r}{m-1}\dbinom{m-1}{r}\dbinom{r}{i}\lambda^{r-i}\kappa^i\bfone^{\ot (m+n+1-r-i)}\\
	    &+\sum_{r=0}^{m-1}\sum_{i=0}^{r}\dbinom{m+n-2-r}{m-1}\dbinom{m-1}{r}\dbinom{r}{i}\lambda^{r+1-i}\kappa^i\bfone^{\ot (m+n-r-i)}\\
	    &+\sum_{r=0}^{m-1}\sum_{i=0}^{r}\dbinom{m+n-2-r}{m-1}\dbinom{m-1}{r}\dbinom{r}{i}\lambda^{r-i}\kappa^{i+1}\bfone^{\ot (m+n-1-r-i)}\\
		={}& \sum_{r=0}^{m}\sum_{i=0}^{r}\dbinom{m+n-1-r}{m}\dbinom{m}{r}\dbinom{r}{i}\lambda^{r-i}\kappa^i\bfone^{\ot (m+n+1-r-i)}\\
		&+\sum_{r=0}^{m}\sum_{i=0}^{r}\dbinom{m+n-1-r}{m-1}\dbinom{m-1}{r}\dbinom{r}{i}\lambda^{r-i}\kappa^i\bfone^{\ot (m+n+1-r-i)} \quad(\text{by the convention $\dbinom{m-1}{m}=0 $})\\
		&+\sum_{r=0}^{m-1}\sum_{i=0}^{r+1}\dbinom{m+n-2-r}{m-1}\dbinom{m-1}{r}\dbinom{r}{i}\lambda^{r+1-i}\kappa^i\bfone^{\ot (m+n-r-i)}\quad(\text{by the convention $\dbinom{r}{r+1}=0 $})\\
		&+\sum_{r=0}^{m-1}\sum_{i=0}^{r+1}\dbinom{m+n-2-r}{m-1}\dbinom{m-1}{r}\dbinom{r}{i-1}\lambda^{r+1-i}\kappa^{i}\bfone^{\ot (m+n-r-i)}\\
		& \quad(\text{by shifting the index $i$ and the convention $\dbinom{r}{0-1}=\dbinom{r}{-1}=0 $})\\
		={}& \sum_{r=0}^{m}\sum_{i=0}^{r}\dbinom{m+n-1-r}{m}\dbinom{m}{r}\dbinom{r}{i}\lambda^{r-i}\kappa^i\bfone^{\ot (m+n+1-r-i)}\\
		&+\sum_{r=0}^{m}\sum_{i=0}^{r}\dbinom{m+n-1-r}{m-1}\dbinom{m-1}{r}\dbinom{r}{i}\lambda^{r-i}\kappa^i\bfone^{\ot (m+n+1-r-i)}\\
		&+\sum_{r=0}^{m-1}\sum_{i=0}^{r+1}\dbinom{m+n-2-r}{m-1}\dbinom{m-1}{r}\dbinom{r+1}{i}\lambda^{r+1-i}\kappa^i\bfone^{\ot (m+n-r-i)}\\
&\quad(\text{by adding the last two terms in the previous equation})\\
		={}& \sum_{r=0}^{m}\sum_{i=0}^{r}\dbinom{m+n-1-r}{m}\dbinom{m}{r}\dbinom{r}{i}\lambda^{r-i}\kappa^i\bfone^{\ot (m+n+1-r-i)}\\
		&+\sum_{r=0}^{m}\sum_{i=0}^{r}\dbinom{m+n-1-r}{m-1}\dbinom{m}{r}\dbinom{r}{i}\lambda^{r-i}\kappa^i\bfone^{\ot (m+n+1-r-i)}\\
& \quad(\text{by shifting the index $r$ and the convention $\dbinom{m-1}{0-1}=\dbinom{m-1}{-1}=0 $})\\
		={}& \sum_{r=0}^{m}\sum_{i=0}^{r}\dbinom{m+n-r}{m}\dbinom{m}{r}\dbinom{r}{i}\lambda^{r-i}\kappa^i\bfone^{\ot (m+n+1-r-i)}\\
		={}& \sum_{r=0}^{\min\{m,n\}}\sum_{i=0}^{r}\dbinom{m+n-r}{m}\dbinom{m}{r}\dbinom{r}{i}\lambda^{r-i}\kappa^i\bfone^{\ot (m+n+1-r-i)}.
\end{align*}
This completes the induction and thus the proof of Eq.~(\mref{eq:powerzero}).

The second equation follows by taking $\ell=r+i$.
\end{proof}

Let
\vsb
$$P_\bfk:\sha_e(\bfk)\to \sha_e(\bfk), \quad \bfone^{\ot n}\mapsto \bfone^{\ot n+1}, n\geq 0,$$
denote the \wmrbo on $\sha_e(\bfk)$. Then together with the product $\dg$ given in Eq.~(\mref{eq:powerzero}), we obtain
\vsb
\begin{coro}
The \wmrb $(\sha_e(\bfk),\dg,P_\bfk)$ is the initial object in the category of commutative \wmrbs. In other words, for any commutative \wmrb $(R,P)$, there exists a unique \wmrb homomorphism $\bar{f}: (\sha_e(\bfk), \dg, P_\bfk) \to (R,P)$.
\end{coro}
\vsb
\begin{proof}
For any given commutative \wmrb $(R,P)$, the existence  and uniqueness of $\bar{f}$ follows from the unique algebra homomorphism  $f:\bfk\to R$, $c\mapsto c1_R$.
\end{proof}

\vsb
\subsection{Diagonally 2-colored Delannoy paths}
\mlabel{ss:del}
We now give a combinatorial interpretation for the coefficients appearing in the product formula in Eq.~\meqref{eq:powerzero}.

Recall~\cite[Proposition~3.2.3]{Gub} that the product of the free commutative Rota-Baxter algebra $\sha(\bfk)$ is given by
\vsc
\begin{equation}\mlabel{eq:mix}
\bfone^{\ot(m+1)}\diamond_\lambda \bfone^{\ot(n+1)}=\sum_{r=0}^{\min{(m,n)}}\dbinom{m+n-r}{m}\dbinom{m}{r}\lambda^r\bfone^{\ot (m+n+1-r)},\quad m,n\geq 0,
\vsb
\end{equation}
where $\diamond_\lambda$ is  the augmented mixable shuffle product of weight $\lambda$. The coefficients have the  combinatorial interpretation by Delannoy paths~\mcite{Gub}. A Delannoy path is a lattice path in the plane $\ZZ^2 $ from $(0, 0)$ to $(m,n)$ using only up steps $(0, 1)$, right steps $(1, 0)$,
and diagonal steps $(1, 1)$. Denote by $D(m,n)$ (resp. $D(m,n,r)$) the number of Delannoy paths from $(0,0)$ to $(m,n)$ (resp. containing exactly $r$ diagonal steps for $0\leq r\leq \min{(m,n)}$). Then
\vsb
\begin{equation}\mlabel{eq:dmnr}
D(m,n,r)=\dbinom{m+n-r}{m}\dbinom{m}{r}.
\vsb
\end{equation}
By Eq.~(\mref{eq:mix}), we get
\vsc
\begin{equation}
D(m,n)=\sum_{r=0}^{\min{(m,n)}}D(m,n,r).
\vsb
\end{equation}
The numbers $D(m,n)$ has the well-known generating function
$$ \sum_{m,n\geq 0} D(m,n)x^my^n=\frac{1}{1-x-y-xy}.$$

Now denote by $E(m,n)$  the  coefficients appeared in Eq.~(\mref{eq:powerzero}). Then
$$E(m,n)=\sum_{r=0}^{\min\{m,n\}}\sum_{i=0}^{r}\dbinom{m+n-r}{m}\dbinom{m}{r}\dbinom{r}{i}
=\sum_{r=0}^{\min\{m,n\}}2^r\dbinom{m+n-r}{m}\dbinom{m}{r}
=\sum_{r=0}^{\min\{m,n\}} 2^r D(m,n,r).
$$
Let $E(m,n,r):=  2^r D(m,n,r)$. Then
\vsb
\begin{equation}\mlabel{eq:emn}
E(m,n)=\sum_{r=0}^{\min\{m,n\}} E(m,n,r).
\vsb
\end{equation}
Direct computation of $E(m,n)$ for the first few values gives the table
$$\begin{array}{ll}
\{E(m,n), 0\leq m\le n\leq , k\geq 0\}=	
\{ 	&1, \\
 &1, 1, \\
 &1, 4, 1, \\
 &1, 7, 7, 1, \\
 &1, 10, 22, 10, 1, \\
 &1, 13, 46, 46, 13, 1, \\
 &\cdots \cdots \}\end{array}
 $$
This sequence coincides with the sequence from Pascal-like triangles and  left-factors of Schr\"oder paths. See \cite[A081577]{Slo} and ~\cite[Example~3]{Bar} for further details. We next give an interpretation of $E(m,n)$ in terms of a diagonally 2-colored version of Delannoy paths.
\vsb

\begin{defn} A {\bf diagonally 2-colored Delannoy path} is a Delannoy path with each diagonal step decorated by two colors.
\end{defn}

By convention, a Delannoy path without diagonals is regarded as a diagonally 2-colored Delannoy path.
Some  diagonally 2-colored Delannoy paths are given in Table~\mref{tab:twodp}, where we use a solid red line $\textcolor{red}{|}$
and a dashed blue line $\,\dashed$ to denote the two types of diagonals.

{\begin{table}[!htp]
\begin{tabular}{ |c|c|c|c|c|c|}
\hline
{ $(m,n)$}&$r=0$&$r=1$&$|\frakD_c(m,n)|$\\
\hline
(1,1)
& \makecell[c]{\dlpa\, \dlpb} & \makecell[c]{\dlpc\, \dlpd}&4\\
\hline
\rule{0pt}{25pt}
(1,2) & \makecell[c]{\dlpe\, \dlpee\,\dlpf}  &\makecell[c]{\dlpg\,\dlph\,\dlpgg\,\dlphh} &7\\
\hline
\rule{0pt}{25pt}
(1,3)&\makecell[c]{\dlphr\,\dlphs\dlpht\dlphu }&\makecell[c]{\dlphv\,\dlphw\,\dlphx\dlphy\,\dlphzz\,\dlphz}      &10\\
\hline
\end{tabular}
\smallskip
\caption{\small Some examples of diagonally 2-colored Delannoy paths from $(0,0)$ to $(m,n)$.}
\mlabel{tab:twodp}
\end{table}
}

Denote by $\frakD_c(m,n)$ (resp. $\frakD_c(m,n,r)$) the set of diagonally 2-colored Delannoy paths from $(0,0)$ to $(m,n)$ (resp. containing exactly $r$ diagonal steps with $0\leq r\leq \min{(m,n)}$).

According to~\cite[A081577]{Slo}, the numbers $|\frakD_c(m,n)|$ has the simple generating function
$$\sum_{m,n\geq 0} |\frakD_c(m,n)|x^m y^n = \frac{1}{1-x-y-2xy}.$$

We further have
\vsb
\begin{equation}\mlabel{eq:frakd}
\frakD_c(m,n)=\bigsqcup_{r\geq0}^{\min{(m,n)}} \frakD_c(m,n,r).
\end{equation}
\begin{lemma}Let $m,n\in\NN$. For any given $0\leq r \leq\min{(m,n)}$, we have
\begin{equation}\mlabel{eq:mnr}
|\frakD_c(m,n,r)|=E(m,n,r).
\end{equation}
\end{lemma}
\begin{proof}
We apply the induction on $m+n\geq 0$ to prove Eq.~(\mref{eq:mnr}). When $m+n=0$, then $m=n=0$, and so $r=0$. Then
\vsb
$$|\frakD_c(m,n,r)|=|\frakD_c(0,0,0)|=1=E(0,0,0)=E(m,n,r).$$
For a given $k\geq 0$, suppose that Eq.~(\mref{eq:mnr}) holds for $0\le m+n\leq k$  and $0\leq r \leq\min{(m,n)}$. Now we consider the case when $m+n=k+1$. Then $k+1\geq 1$.
If one of $m,n$ is zero, without loss of generality, we set $m=0$. So $r=0$ and
$$|\frakD_c(m,n,r)|=|\frakD_c(0,n,0)|=1=E(0,n,0)=E(m,n,r).$$
Let $m,n\geq 1$. Then we distinguish two cases.

\smallskip
\noindent
{\bf Case 1. $r=0$}. Then by the definition of Delannoy paths, we have
$$|\frakD_c(m,n,0)|=|\frakD_c(m-1,n,0)|+|\frakD_c(m,n-1,0)|.$$
Thus
\vsb
\begin{align*}
	|\frakD_c(m,n,0)|
	={}& |\frakD_c(m-1,n,0)|+|\frakD_c(m,n-1,0)|\\
	={}& E(m-1,n,0)+E(m,n-1,0)\quad(\text{by the induction hypothesis})\\
	={}& 2^0\dbinom{m+n-1}{m-1}\dbinom{m-1}{0}+2^0\dbinom{m+n-1}{m}\dbinom{m}{0}\\
	={}& \dbinom{m+n}{m}
	={} E(m,n,0).
\end{align*}
\smallskip
\noindent
{\bf Case 2. $1\le r\le \min{(m,n)} $}. Then by the definition of diagonally 2-colored Delannoy paths, we have
\vsb
$$|\frakD_c(m,n,r)|=|\frakD_c(m-1,n,r)|+|\frakD_c(m,n-1,r)|+2|\frakD_c(m-1,n-1,r-1)|.$$
Then we obtain
\vsb
\begin{align*}
	&|\frakD_c(m,n,r)|\\
	={}& |\frakD_c(m-1,n,r)|+|\frakD_c(m,n-1,r)|+2*|\frakD_c(m-1,n-1,r-1)|\\
	={}& E(m-1,n,r)+E(m,n-1,r)+2*E(m-1,n-1,r-1) \quad(\text{by the induction hypothesis})\\
	={}& 2^r\dbinom{m+n-1-r}{m-1}\dbinom{m-1}{r}+2^r\dbinom{m+n-1-r}{m}\dbinom{m}{r}+2*2^{r-1}\dbinom{m+n-2-(r-1)}{m-1}
	\dbinom{m-1}{r-1}\\
	={}& 2^r\dbinom{m+n-r}{m}\dbinom{m}{r}\\
	={}& E(m,n,r).
\end{align*}
This completes the induction and hence the proof of the lemma.
\end{proof}

\begin{prop}For $m,n\geq 0$, we have
\vsb
\begin{equation}\mlabel{eq:dmn}
|\frakD_c(m,n)|=E(m,n).
\end{equation}
\mlabel{pp:del}
\end{prop}
\begin{proof}
By Eqs.~(\mref{eq:emn}),~(\mref{eq:frakd}) and ~(\mref{eq:mnr}), we obtain
$$\hspace{2cm} |\frakD_c(m,n)|=\sum_{r\geq0}^{\min{(m,n)}} |\frakD_c(m,n,r)|
=\sum_{r\geq0}^{\min{(m,n)}}E(m,n,r)
=E(m,n). \hspace{2.5cm}\qedhere
$$
\end{proof}

\vsd

\section{The Hopf algebra structure on free commutative \wmrbs}
\mlabel{sec: Hopfcomm}

In this section,  we will equip the free commutative \wmrb $\sha_e(A):=(\sha_e(A), \dg, \pg, j_A)$ with a Hopf algebra structure, when the generating algebra $A$ is a  connected filtered bialgebra.  In Section~\mref{subsec:bistucture}, we first describe a coalgebra structure on the free commutative \wmrb $\sha_e(A)$.
Then Section~\mref{subsec:Hopfstructure} provides a Hopf algebraic structure on $\sha_e(A)$.
An explicit comultiplication formula in the special case when $A=\bfk$ is given at the end of the section.
\vsb

\subsection{The bialgebra structure from universal properties}
\mlabel{subsec:bistucture}
A {\bf bialgebra} is a quintuple $A:=(A,m_A,\mu_A,\Delta_A,\vep_A)$ where $(A,m_A,\mu_A)$  is an algebra and $(A,\Delta_A,\vep_A)$
is a coalgebra such that $\Delta_A: A\to A\ot A$  and $\vep_A: A\to \bfk$ are morphisms of algebras.
We first construct a comultiplication $\delg$ on $\sha_e(A)$  and a counit $\vg$ on $\sha_e(A)$ by using the universal property of the free \wmrb $\sha_e(A)$,  the comultiplication $\Delta_A$ and the counit $\vep_A$ of $A$.

\begin{lemma}Let $A:=(A,m_A,\mu_A,\Delta_A,\vep_A)$ be a bialgebra. Let $\sha_e(A):=(\sha_e(A), \dg, \pg, j_A)$ be the free commutative \wmrb of weight $(\lambda, \kappa)$ on $A$. Let $\mu\in\bfk$, with $\mu$ being a root of $t^2-\lambda t+\kappa$. There exists a unique homomorphism $\vg:\sha_e(A)\to \bfk$ of \wmrb of weight $(\lambda, \kappa)$ such that
\begin{equation}\mlabel{eq:vgj}
\vg\circ j_A=\vep_A\quad\text{and}\quad \vg\circ \pg= -\mu\id\circ\vg.
\end{equation}
\mlabel{lem:vg}
\end{lemma}

\begin{proof}Since $P:=-\id_\bfk$ is a Rota-Baxter operator of weight $1$ on $\bfk$, we have $\tilde{P}:=-\id_\bfk-P=0$. By Corollary~\mref{co:genf}~\eqref{it:genf2} and $\mu$ being the root of $t^2-\lambda t+\kappa$, we obtain $Q:=-\mu\id_\bfk$ is an \wmrbo of weight $(\lambda,\kappa)$ on $\bfk$. Then the result follows from  the universal property of $\sha_e(A)$.
\end{proof}

In the following, we use $\bt$ to denote the tensor product  between $\sha_e(A)$ to distinguish from the usual tensor product in $\sha_e(A)$. Denote by $\bug$ the multiplication in the commutative algebra $\sha_e(A)\bt \sha_e(A)$. First, define a linear operator
\begin{equation}
\bar{P}:\sha_e(A)\bt\sha_e(A)\to \sha_e(A)\bt\sha_e(A), \fraka\bt\frakb\mapsto \pg(\fraka)\bt\vg(\frakb)1_A+\mu\fraka\bt \vg(\frakb)1_A+\fraka\bt\pg(\frakb).
\mlabel{eq:barp}
\end{equation}
\begin{lemma}\mlabel{lem:barp}
The linear operator $\bar{P}$ is an \wmrbo on $\sha_e(A)\bt \sha_e(A)$ of weight $(\lambda,\kappa)$.
\end{lemma}
\begin{proof}
	Let $\fraka,\frakb,\frakc,\frakd,\in\sha_e(A)$. By Eq.~(\mref{eq:barp}), we get
	\begin{eqnarray*}
		&&\bar{P}(\fraka\bt\frakb)\bug\bar{P}(\frakc\bt\frakd)\\
		&=&\bigg(\pg(\fraka)\bt\vg(\frakb)1_A+\mu\fraka\bt \vg(\frakb)1_A+\fraka\bt\pg(\frakb)\bigg)\bug\bigg(\pg(\frakc)\bt\vg(\frakd)1_A+\mu\frakc\bt \vg(\frakd)1_A+\frakc\bt\pg(\frakd)\bigg)\\
		&=&\Big(\pg(\fraka)\dg\pg(\frakc)\Big)\bt\Big(\vg(\frakb)\vg(\frakd)1_A\Big)
		+\mu\Big(\pg(\fraka)\dg\frakc\Big)\bt\Big(\vg(\frakb)\vg(\frakd)1_A\Big)+\Big(\pg(\fraka)\dg\frakc\Big)\bt\Big(\vg(\frakb)\pg(\frakd)\Big)\\
		&&+\mu\Big(\fraka\dg\pg(\frakc)\Big)\bt\Big(\vg(\frakb)\vg(\frakd)1_A\Big)
		+\mu^2\Big(\fraka\dg\frakc\Big)\bt\Big(\vg(\frakb)\vg(\frakd)1_A\Big)+\mu\Big(\fraka\dg\frakc\Big)\bt\Big(\vg(\frakb)\pg(\frakd)\Big)\\
		&&+\Big(\fraka\dg\pg(\frakc)\Big)\bt\Big(\vg(\frakd)\pg(\frakb)\Big)
		+\mu\Big(\fraka\dg\frakc\Big)\bt\Big(\vg(\frakd)\pg(\frakb)\Big)+\Big(\fraka\dg\frakc\Big)\bt\Big(\pg(\frakb)\dg\pg(\frakd)\Big)\\
		&=&\pg(\fraka\dg\pg(\frakc))\bt\vg(\frakb)\vg(\frakd)1_A+\pg(\pg(\fraka)\dg\frakc)\bt\vg(\frakb)\vg(\frakd)1_A+\lambda \pg(\fraka\dg\frakc)\bt\vg(\frakb)\vg(\frakd)1_A\\
		&&+\kappa(\fraka\dg\frakc)\bt\vg(\frakb)\vg(\frakd)1_A
		+\mu(\pg(\fraka)\dg\frakc)\bt\vg(\frakb)\vg(\frakd)1_A
		+(\pg(\fraka)\dg\frakc)\bt\vg(\frakb)\pg(\frakd)\\
		&&+\mu(\fraka\dg\pg(\frakc))\bt\vg(\frakb)\vg(\frakd)1_A
		+\lambda\mu\Big(\fraka\dg\frakc\Big)\bt\Big(\vg(\frakb)\vg(\frakd)1_A\Big)
		-\kappa\Big(\fraka\dg\frakc\Big)\bt\Big(\vg(\frakb)\vg(\frakd)1_A\Big)\\
		&&+\mu(\fraka\dg\frakc)\bt\vg(\frakb)\pg(\frakd)
		+(\fraka\dg\pg(\frakc))\bt\vg(\frakd)\pg(\frakb)
		+\mu(\fraka\dg\frakc)\bt\vg(\frakd)\pg(\frakb)\\
		&&+(\fraka\dg\frakc)\bt\pg(\frakb\dg\pg(\frakd))+(\fraka\dg\frakc)\bt\pg(\pg(\frakb)\dg\frakd)
		+\lambda(\fraka\dg\frakc)\bt\pg(\frakb\dg\frakd)\\
		&&+\kappa(\fraka\dg\frakc)\bt(\frakb\dg\frakd).\\
		&&\big(\text{by $\pg$ being an \wmrbo of weight $(\lambda,\kappa)$} \\
		&& \text{ and replacing $\mu^2$ by $\lambda\mu-\kappa$ since $\mu^2-\lambda \mu+\kappa=0$}\big)\\
		&=&\pg(\fraka\dg\pg(\frakc))\bt\vg(\frakb)\vg(\frakd)1_A+\pg(\pg(\fraka)\dg\frakc)\bt\vg(\frakb)\vg(\frakd)1_A+\lambda \pg(\fraka\dg\frakc)\bt\vg(\frakb)\vg(\frakd)1_A\\
		&&+\mu(\pg(\fraka)\dg\frakc)\bt\vg(\frakb)\vg(\frakd)1_A
		+(\pg(\fraka)\dg\frakc)\bt\vg(\frakb)\pg(\frakd)
		+\mu(\fraka\dg\pg(\frakc))\bt\vg(\frakb)\vg(\frakd)1_A\\
		&&+\lambda\mu\Big(\fraka\dg\frakc\Big)\bt\Big(\vg(\frakb)\vg(\frakd)1_A\Big)
		+\mu(\fraka\dg\frakc)\bt\vg(\frakb)\pg(\frakd)
		+(\fraka\dg\pg(\frakc))\bt\vg(\frakd)\pg(\frakb)\\
		&&+\mu(\fraka\dg\frakc)\bt\vg(\frakd)\pg(\frakb)
		+(\fraka\dg\frakc)\bt\pg(\frakb\dg\pg(\frakd))+(\fraka\dg\frakc)\bt\pg(\pg(\frakb)\dg\frakd)\\
		&&+\lambda(\fraka\dg\frakc)\bt\pg(\frakb\dg\frakd)
		+\kappa(\fraka\dg\frakc)\bt(\frakb\dg\frakd)\\
		&&\Big(\text{by collecting the terms with } \kappa\big(\fraka\dg\frakc\big)\bt\big(\vg(\frakb)\vg(\frakd)1_A\big)\Big).
	\end{eqnarray*}
	On the other hand, we have
	\begin{eqnarray*}
		&&\bar{P}((\fraka\bt\frakb)\bug\bar{P}(\frakc\bt\frakd))+\bar{P}(\bar{P}(\fraka\ot\frakb)\bug(\frakc\ot\frakd))
		+\lambda\bar{P}((\fraka\bt\frakb)\bug(\frakc\bt\frakd))+\kappa(\fraka\bt\frakb)\bug(\frakc\bt\frakd)\\
		&=&\bar{P}\Big((\fraka\bt\frakb)\bug(\pg(\frakc)\bt\vg(\frakd)1_A+\mu\frakc\bt\vg(\frakd)1_A+\frakc\bt\pg(\frakd))\Big)\\
		&&+\bar{P}\Big(\big(\pg(\fraka)\bt\vg(\frakb)1_A+\mu\fraka\bt\vg(\frakb)1_A+\fraka\bt\pg(\frakb)\big)\bug(\frakc\ot\frakd)\Big)\\
		&&+\lambda\bar{P}((\fraka\bt\frakb)\bug(\frakc\bt\frakd))+\kappa(\fraka\bt\frakb)\bug(\frakc\bt\frakd)\\
		&=&\bar{P}\Big((\fraka\dg\pg(\frakc))\bt(\frakb\dg\vg(\frakd)1_A)+\mu(\fraka\dg\frakc)\bt(\frakb\dg\vg(\frakd)1_A)
		+(\fraka\dg\frakc)\bt(\frakb\dg\pg(\frakd))\Big)\\
		&&+\bar{P}\Big((\pg(\fraka)\dg\frakc)\bt(\vg(\frakb)1_A\dg\frakd)
		+\mu(\fraka\dg\frakc)\bt(\vg(\frakb)1_A\dg\frakd)+(\fraka\dg\frakc)\bt(\pg(\frakb)\dg\frakd)\Big)\\
		&&+\lambda\pg(\fraka\dg\frakc)\bt\vg(\frakb\dg\frakd)1_A+\lambda\mu(\fraka\dg\frakc)\bt\vg(\frakb\dg\frakd)1_A+\lambda(\fraka\dg\frakc)\bt\pg(\frakb\dg\frakd)\\
		&&+\kappa(\fraka\dg\frakc)\bt(\frakb\dg\frakd)\\
		&=&\pg(\fraka\dg\pg(\frakc))\bt\vg(\frakb)\vg(\frakd)1_A+\mu(\fraka\dg\pg(\frakc))\bt\vg(\frakb)\vg(\frakd)1_A
		+(\fraka\dg\pg(\frakc))\bt\vg(\frakd)\pg(\frakb)\\
		&&+\mu\pg(\fraka\dg\frakc)\bt\vg(\frakb)\vg(\frakd)1_A
		+\mu^2(\fraka\dg\frakc)\bt\vg(\frakb)\vg(\frakd)1_A+\mu(\fraka\dg\frakc)\bt\vg(\frakd)\pg(\frakb)\\
		&&-\mu\pg(\fraka\dg\frakc)\bt\vg(\frakb)\vg(\frakd))1_A
		-\mu^2(\fraka\dg\frakc)\bt\vg(\frakb)\vg(\frakd)1_A+(\fraka\dg\frakc)\bt\pg(\frakb\dg\pg(\frakd))\\
		&&+\pg(\pg(\fraka)\dg\frakc)\bt\vg(\frakd)\vg(\frakb)1_A
		+\mu(\pg(\fraka)\dg\frakc)\bt\vg(\frakb)\vg(\frakd)1_A+(\pg(\fraka)\dg\frakc)\bt\vg(\frakb)\pg(\frakd)\\
		&&+\mu\pg(\fraka\dg\frakc)\bt\vg(\frakb)\vg(\frakd)1_A
		+\mu^2(\fraka\dg\frakc)\bt\vg(\frakb)\vg(\frakd)1_A+\mu(\fraka\dg\frakc)\bt\vg(\frakb)\pg(\frakd)\\
		&&-\mu\pg(\fraka\dg\frakc)\bt\vg(\frakb)\vg(\frakd)1_A
		-\mu^2(\fraka\dg\frakc)\bt\vg(\frakb)\vg(\frakd)1_A+(\fraka\dg\frakc)\bt\pg(\pg(\frakb)\dg\frakd)\\
		&&+\lambda\pg(\fraka\dg\frakc)\bt\vg(\frakb)\vg(\frakd)1_A+\lambda\mu(\fraka\dg\frakc)\bt\vg(\frakb)\vg(\frakd)1_A+\lambda(\fraka\dg\frakc)\bt\pg(\frakb\dg\frakd)\\
		&&+\kappa(\fraka\dg\frakc)\bt(\frakb\dg\frakd)\\
		&&\quad(\text{by Eq.~(\mref{eq:barp}), Lemma~\mref{lem:vg} and $\vg$ being an algebra homomorphism})\\
		&=&\pg(\fraka\dg\pg(\frakc))\bt\vg(\frakb)\vg(\frakd)1_A+\mu(\fraka\dg\pg(\frakc))\bt\vg(\frakb)\vg(\frakd)1_A
		+(\fraka\dg\pg(\frakc))\bt\vg(\frakd)\pg(\frakb)\\
		&&+\mu(\fraka\dg\frakc)\bt\vg(\frakd)\pg(\frakb)+(\fraka\dg\frakc)\bt\pg(\frakb\dg\pg(\frakd))+\pg(\pg(\fraka)\dg\frakc)\bt\vg(\frakd)\vg(\frakb)1_A\\
		&&+\mu(\pg(\fraka)\dg\frakc)\bt\vg(\frakb)\vg(\frakd)1_A+(\pg(\fraka)\dg\frakc)\bt\vg(\frakb)\pg(\frakd)+\mu(\fraka\dg\frakc)\bt\vg(\frakb)\pg(\frakd)\\
		&&+(\fraka\dg\frakc)\bt\pg(\pg(\frakb)\dg\frakd)+\lambda\pg(\fraka\dg\frakc)\bt\vg(\frakb)\vg(\frakd)1_A+\lambda\mu(\fraka\dg\frakc)\bt\vg(\frakb)\vg(\frakd)1_A\\
		&&+\lambda(\fraka\dg\frakc)\bt\pg(\frakb\dg\frakd)+\kappa(\fraka\dg\frakc)\bt(\frakb\dg\frakd)\\
		&&\Big(\text{by collecting the terms with } \mu\pg(\fraka\dg\frakc)\bt\vg(\frakb)\vg(\frakd)1_A \text{ and } \mu^2(\fraka\dg\frakc)\bt\vg(\frakb)\vg(\frakd)1_A\Big).
	\end{eqnarray*}
	Then the $i$-th term in the expansion of $\bar{P}(\fraka\bt\frakb)\bug\bar{P}(\frakc\bt\frakd)$ matches with the $\sigma(i)$-th term in the expansion of $\bar{P}((\fraka\bt\frakb)\bug\bar{P}(\frakc\bt\frakd))+\bar{P}(\bar{P}(\fraka\ot\frakb)\bug(\frakc\ot\frakd))
	+\lambda\bar{P}((\fraka\bt\frakb)\bug(\frakc\bt\frakd))+\kappa(\fraka\bt\frakb)\bug(\frakc\bt\frakd)$, where the permutation $\sigma\in \sum_{14}$ is
	\begin{equation*}
		\begin{pmatrix}
			i\\ \sigma(i)
		\end{pmatrix}
		=
		\left ( \begin{array}{cccccccccccccccc}
			1 & 2 & 3 & 4 & 5 & 6 & 7 & 8 & 9 & 10 & 11 & 12 & 13 & 14 \\
			1 & 6 & 11 & 7 & 8 & 2 & 12 & 9 & 3 & 4 & 5 & 10 & 13 & 14
		\end{array} \right).
	\end{equation*}
	Thus $\bar{P}$ is an \wmrbo of weight $(\lambda,\kappa)$ on $\sha_e(A)\bt\sha_e(A)$.
\end{proof}

By Lemma~\mref{lem:barp}, the pair $(\sha_e(A)\ot\sha_e(A),\bar{P})$ is an \wmrb of weight $(\lambda,\kappa)$. Then by applying the universal property of $\sha_e(A)$, we obtain
\begin{lemma}
There exists a unique \wmrb homomorphism
\begin{equation*}
\delg:\sha_e(A)\to\sha_e(A)\bt\sha_e(A)	
\end{equation*}
such that
\vsb
\begin{equation}
\delg\circ j_A=\Delta_A\quad\text{and}\quad \delg\circ \pg=\bar{P}\circ \delg.
\mlabel{eq:delg}
\end{equation}
\mlabel{lem:delg}
\end{lemma}
\vsb
\begin{lemma}With the same notations as above, $\vg$ is idempotent for the convolution product $\ast$:
\begin{equation}
\vg\ast \vg=\vg.
\mlabel{eq:cvg}
\end{equation}
\end{lemma}
\begin{proof}
By Eqs.~\meqref{eq:vgj} and~\meqref{eq:delg}, $\vg \circ j_A=\vep_A$ and $\delg\circ j_A=\Delta_A$.
We next verify that
\begin{equation}
m_\bfk(\vg\ot\vg)\Delta_A=\vep_A.
\mlabel{eq:bfv}
\end{equation}
For $a\in A$, we have
\begin{eqnarray*}
(m_\bfk(\vg\ot\vg)\Delta_A)(a)&=&m_\bfk(\vg\ot\vg)(\sum_{a}a_{(1)}\ot a_{(2)})\\
&=&m_\bfk(\sum_{a}\vep_A(a_{(1)})\ot \vep_A(a_{(2)}))\quad(\text{by Eq.~\meqref{eq:vgj}})\\
&=&m_\bfk(\vep_A\ot \id_\bfk)(\id_A\ot \vep_A)\Delta(a)\\
&=&m_\bfk(\vep_A\ot \id_\bfk)(a\ot 1_\bfk)\quad(\text{by the right counicity of $\vep_A$})\\
&=&m_\bfk(\vep_A\ot \id_\bfk)(a\ot 1_\bfk)\\
&=&\vep_A(a),
\end{eqnarray*}
proving Eq.~(\mref{eq:bfv}). Thus from Eq.~\meqref{eq:delg}, we obtain $(m_\bfk(\vg\ot\vg)\delg)\circ j_A=\vep_A$.

Furthermore, $m_\bfk(\vg\ot\vg)\Delta_A$ is an algebra homomorphism from $\sha_e(A)$ to $\bfk$. For $\fraka\in\sha_e(A)$, we get
\begin{eqnarray*}
&&(m_\bfk(\vg\ot\vg)\delg)\pg(\fraka)\\
&=&(m_\bfk(\vg\ot\vg)(\bar{P}\delg(\fraka))\quad(\text{by Eq.~\meqref{eq:delg}})\\
&=&(m_\bfk(\vg\ot\vg)\sum_{\fraka}\pg(\fraka_{(1)})\ot\vg(\fraka_{(2)})1_A+ \mu\fraka_{(1)}\ot \vg(\fraka_{(2)})1_A+\fraka_{(1)}\ot \pg(\fraka_{(2)})\quad(\text{by Eq.~\meqref{eq:barp}})\\
&=&m_\bfk\Big(\sum_{\fraka}-\mu\vg(\fraka_{(1)})\ot\vg(\fraka_{(2)})1_\bfk+ \mu\vg(\fraka_{(1)})\ot \vg(\fraka_{(2)})1_\bfk-\mu\vg(\fraka_{(1)})\ot \vg(\fraka_{(2)})\Big)\quad(\text{by Eq.~\meqref{eq:vgj}})\\
&=&(-\mu\id)(m_\bfk(\vg\ot\vg))\delg(\fraka).
\end{eqnarray*}
Thus $m_\bfk(\vg\ot\vg)\delg$ is an \wmrb homomorphism of weight $(\lambda,\kappa)$. Let $m_\bfk$  be the multiplication of $\bfk$.  By the uniqueness of $\vg$, we obtain $\vg\ast\vg=m_\bfk(\vg\ot\vg)\delg=\vg$. This proves the desired result.
\end{proof}
We further have the following conclusion.
\begin{lemma}Let $\delg$ and $\vg$ be as above.
Define a linear operator $\tilde{P}$ on $ \sha_e(A)\bt\sha_e(A)\bt\sha_e(A)$ by
\begin{equation}
\tilde{P}(\fraka\bt\frakb\bt\frakc)=\Big(\pg(\fraka)+\mu\fraka\Big)\bt\vg(\frakb)1_A\bt\vg(\frakc)1_A+\fraka\bt\Big(\pg(\frakb)
+\mu\frakb\Big)\bt\vg(\frakc)1_A+\fraka\bt\frakb\bt\pg(\frakc).
\mlabel{eq:tildp}
\end{equation}
Then $\tilde{P}$ is an \wmrbo on $ \sha_e(A)\bt\sha_e(A)\bt\sha_e(A)$.
Furthermore, $(\id\ot \delg)\delg$ and $(\delg\ot\id)\delg$ are \wmrb homomorphisms from $(\sha_e(A),\pg)$ to $\Big(\sha_e(A)\bt\sha_e(A)\bt\sha_e(A),\tilde{P}\Big)$.
\mlabel{lem:tilp}
\end{lemma}
\begin{proof}
The first statement follows from a direct computation.
For the second statement, let $\fraka\in\sha_e(A)$. We have
\begin{eqnarray*}
&&((\id\ot \delg)\delg\pg)(\fraka)\\
&=&(\id\ot\delg)\bar{P}(\delg(\fraka))\quad(\text{by Eq.~(\mref{eq:delg})})\\
&=&(\id\ot\delg)\bar{P}\sum_{\fraka}\fraka_{(1)}\bt\fraka_{(2)}\\
&=&(\id\ot\delg)\sum_{\fraka}\bar{P}(\fraka_{(1)}\bt\fraka_{(2)})\\
&=&(\id\ot\delg)\sum_{\fraka}\Big(\pg(\fraka_{(1)})\bt\vg(\fraka_{(2)})1_A+\mu\fraka_{(1)}\bt\vg(\fraka_{(2)})1_A+\fraka_{(1)}\bt\pg(\fraka_{(2)})\Big)\\
&=&\sum_{\fraka}\Big(\pg(\fraka_{(1)})\bt\vg(\fraka_{(2)})1_A\bt1_A
+\mu\fraka_{(1)}\bt\vg(\fraka_{(2)})1_A\bt1_A+\fraka_{(1)}\bt\delg(\pg(\fraka_{(2)}))\Big)\\
&&\quad(\text{by $\delg$ being an algebra algebra homomorphism})\\
&=&\sum_{\fraka}\Big(\pg(\fraka_{(1)})\bt\vg(\fraka_{(2)})1_A\bt1_A
+\mu\fraka_{(1)}\bt\vg(\fraka_{(2)})1_A\bt1_A\Big)
+\sum_{\fraka}\sum_{\fraka_{(2)}}\fraka_{(1)}\bt\bar{P}\Big(\fraka_{(21)}\bt\fraka_{(22)}\Big)\\
&=&\sum_{\fraka}\Big(\pg(\fraka_{(1)})\bt\vg(\fraka_{(2)})1_A\bt1_A
+\mu\fraka_{(1)}\bt\vg(\fraka_{(2)})1_A\bt1_A\Big)\\
&&+\sum_{\fraka}\sum_{\fraka_{(2)}}\fraka_{(1)}\bt\Big(\pg(\fraka_{(21)})\bt\vg(\fraka_{(22)})1_A
+\mu\fraka_{(21)}\bt\vg(\fraka_{(22)})1_A+\fraka_{(21)}\bt\pg(\fraka_{(22)})\Big).
\end{eqnarray*}
On the other hand,
\begin{eqnarray*}
&&(\tilde{P}(\id\ot\delg)\delg)(\fraka)\\
&=&\tilde{P}(\id\ot\delg)(\sum_{\fraka}\fraka_{(1)}\bt\fraka_{(2)})\\
&=&\tilde{P}(\sum_{\fraka}\sum_{\fraka_{(2)}}\fraka_{(1)}\bt\fraka_{(21)}\bt\fraka_{(22)})\\
&=&\sum_{\fraka}\sum_{\fraka_{(2)}}(\pg(\fraka_{(1)})+\mu\fraka_{(1)})\bt\vg(\fraka_{(21)})1_A\bt\vg(\fraka_{(22)})1_A\\
&&+\sum_{\fraka}\sum_{\fraka_{(2)}}\Big(\fraka_{(1)}\bt(\pg(\fraka_{(21)})+\mu\fraka_{(21)})\bt\vg(\fraka_{(22)})1_A
+\fraka_{(1)}\bt\fraka_{(21)}\bt\pg(\fraka_{(22)})\Big)\quad(\text{by Eq.~(\mref{eq:tildp})})\\
&=&\sum_{\fraka}(\pg(\fraka_{(1)})+\mu\fraka_{(1)})\bt(\sum_{\fraka_{(2)}}\vg(\fraka_{(21)})\vg(\fraka_{(22)}))1_A\bt 1_A\\
&&+\sum_{\fraka}\sum_{\fraka_{(2)}}\Big(\fraka_{(1)}\bt(\pg(\fraka_{(21)})+\mu\fraka_{(21)})\bt\vg(\fraka_{(22)})1_A
+\fraka_{(1)}\bt\fraka_{(21)}\bt\pg(\fraka_{(22)})\Big)\\
&=&\sum_{\fraka}(\pg(\fraka_{(1)})+\mu\fraka_{(1)})\bt\vg(\fraka_{(2)})1_A\bt 1_A\\
&&+\sum_{\fraka}\sum_{\fraka_{(2)}}\Big(\fraka_{(1)}\bt(\pg(\fraka_{(21)})+\mu\fraka_{(21)})\bt\vg(\fraka_{(22)})1_A
+\fraka_{(1)}\bt\fraka_{(21)}\bt\pg(\fraka_{(22)})\Big)\quad(\text{by Eq.~(\mref{eq:cvg})}).
\end{eqnarray*}
Thus $((\id\ot \delg)\delg\pg)(\fraka)=(\tilde{P}(\id\ot\delg)\delg)(\fraka)$, and so $(\id\ot \delg)\delg$ is an \wmrb homomorphism.
By the same argument as above, we can prove that  $(\delg\ot\id)\delg$ is also an \wmrb homomorphism.
\end{proof}

We give another preparation before presenting our main result on bialgebras.
\begin{lemma}
Let $Q$ be an \wmrbo  of weight $(\lambda,\kappa)$ on $\sha_e(A)$. Then  $\id\bt Q$ and $Q\bt\id$ are \wmrbos of weight $(\lambda,\kappa)$ on $\sha_e(A)\bar \ot \sha_e(A)$.
\mlabel{lem:grbo}
\end{lemma}
\begin{proof}
We only consider the operator $\id\bt Q$ since the argument for the other operator is the same.
We only need to verify that $\id\bt Q$ is an \wmrbo for pure tensors $\fraka_1\bt\fraka_2$ and $\frakb_1\bt\frakb_2$ in $\sha_e(A)\bar \ot \sha_e(A)$ with $\fraka_1, \fraka_2, \frakb_1, \frakb_2\in \sha_e(A)$.   For this, we check
\begin{align*}
&(\id\bt Q)(\fraka_1\bt\fraka_2)\bug(\id\bt Q)(\frakb_1\bt\frakb_2)\\
={}& \big(\fraka_1\bt Q(\fraka_2)\big)\bug\big(\frakb_1\bt Q(\frakb_2)\big)\\
={}& (\fraka_1\dg\frakb_1)\bt\big(Q(\fraka_2)\dg Q(\frakb_2)\big)\\
={}& (\fraka_1\dg\frakb_1)\bt\Big(Q\big(Q(\fraka_2)\dg\frakb_2\big)+Q\big(\fraka_2\dg Q(\frakb_2)\big)+\lambda Q(\fraka_2\dg\frakb_2)+\kappa\fraka_2\dg\frakb_2\Big).
\end{align*}
On the other hand,
\begin{align*}
&(\id\bt Q)\Big(\big((\id\bt Q)(\fraka_1\bt\fraka_2)\big)\bug(\frakb_1\bt\frakb_2)\Big)+(\id\bt Q)\Big((\fraka_1\bt\fraka_2)\bug\big((\id\bt Q)(\frakb_1\bt\frakb_2)\big)\Big)\\
&+\lambda(\id\bt Q)\big((\fraka_1\bt\fraka_2)\bug(\frakb_1\bt\frakb_2)\big)+\kappa\big((\fraka_1\bt\fraka_2)\bug(\frakb_1\bt\frakb_2)\big)\\
={}& (\id\bt Q)\Big((\fraka_1\dg\frakb_1)\bt\big(Q(\fraka_2)\dg\frakb_2\big)\Big)+(\id\bt Q)\Big((\fraka_1\dg\frakb_1)\bt\big(\fraka_2\dg Q(\frakb_2)\big)\Big)\\
&+\lambda(\id\bt Q)\big((\fraka_1\dg\frakb_1)\bt(\fraka_2\dg\frakb_2)\big)+\kappa\big((\fraka_1\dg\frakb_1)\bt(\fraka_2\dg\frakb_2)\big)\\
={}& (\fraka_1\dg\frakb_1)\bt Q\big(Q(\fraka_2)\dg\frakb_2\big)+(\fraka_1\dg\frakb_1)\bt Q\big(\fraka_2\dg Q(\frakb_2)\big)\\
&+\lambda(\fraka_1\dg\frakb_1)\bt Q(\fraka_2\dg\frakb_2)+\kappa(\fraka_1\dg\frakb_1)\bt (\fraka_2\dg\frakb_2)\\
={}& (\fraka_1\dg\frakb_1)\bt\Big(Q\big(Q(\fraka_2)\dg\frakb_2\big)+Q\big(\fraka_2\dg Q(\frakb_2)\big)+\lambda Q(\fraka_2\dg\frakb_2)+\kappa\fraka_2\dg\frakb_2\Big).
\end{align*}
This gives what we need.
\end{proof}

Define a $\bfk$-linear map
\vsb
$$\mg:\bfk\to\sha_e(A), \quad  c\mapsto c1_A, c\in \bfk.$$
Then one can prove that $\mg$ is a unit on $\sha_e(A)$ by a direct computation.
\begin{theorem}Let $\mg$, $\vg$ and $\delg$ be as above. Then $\sha_e(A):=(\sha_e(A), \dg, \mg,\delg,\vg)$ forms a bialgebra.
\end{theorem}
\vsb
\begin{proof}
It suffices to prove the counicity of $\vg$ and the coassociativity of $\delg$.
We first verify that $\vg$ satisfies the counicity, that is,
\begin{equation}\mlabel{eq:vgt}
(\vg\ot\id)\delg=\beta_\ell\quad\text{and}\quad(\id\ot\vg)\delg=\beta_r,
\end{equation}
where
\vsb
$$\beta_\ell:\sha_e(A)\to\bfk\ot\sha_e(A), \fraka\mapsto 1_\bfk\ot\fraka \ \text{ and } \beta_r:\sha_e(A)\to\sha_e(A)\ot \bfk,\fraka\mapsto \fraka\ot 1_\bfk$$
are the natural algebra isomorphisms. Let
$$\phi:=\beta_\ell^{-1}(\vg\ot\id)\delg\quad\text{and}\quad\psi:=\beta_r^{-1}(\id\ot\vg)\delg.$$
Then $\phi:\sha_e(A)\to \sha_e(A)$ and $\psi:\sha_e(A)\to\sha_e(A)$ are algebra homomorphisms. Further,
\begin{eqnarray*}
(\phi\pg)(\fraka)&=&\beta_\ell^{-1}(\vg\ot\id)\delg(\pg(\fraka))\\
&=&\beta_\ell^{-1}(\vg\ot\id)\bar{P}(\delg(\fraka))\\
&=&\beta_\ell^{-1}(\vg\ot\id)\sum_{\fraka}\Big(\pg(\fraka_{(1)})\bt\vg(\fraka_{(2)})1_A
+\mu\fraka_{(1)}\bt\vg(\fraka_{(2)})1_A+\fraka_{(1)})\bt\pg(\fraka_{(2)})\Big)\\
&=&\beta_\ell^{-1}\sum_{\fraka}\Big(\vg(\pg(\fraka_{(1)}))\bt\vg(\fraka_{(2)})1_A
+\mu\vg(\fraka_{(1)})\bt\vg(\fraka_{(2)})1_A)+\vg(\fraka_{(1)})\bt\pg(\fraka_{(2)})\Big)\\
&=&\beta_\ell^{-1}\sum_{\fraka}\vg(\fraka_{(1)})\bt\pg(\fraka_{(2)})\quad(\text{by $\vg\circ \pg= -\mu\id\circ\vg$ in Eq.~(\mref{eq:vgj})})\\
&=&\sum_{\fraka}\vg(\fraka_{(1)})\pg(\fraka_{(2)})\\
&=&\pg\Big(\sum_{\fraka}\vg(\fraka_{(1)})\fraka_{(2)}\Big)\\
&=&\Big(\pg\beta_\ell^{-1}(\vg\ot\id)\delg\Big)(\fraka)\\
&=&(\pg\phi)(\fraka).
\end{eqnarray*}
Thus $\phi\pg=\pg\phi$, and so $\phi$ is an \wmrb homomorphism.  In the same way, we verify that $\psi$ is also an \wmrb homomorphism.
By the universal property of $\sha_e(A)$, we obtain $ \phi=\psi=\id$, proving Eq.~(\mref{eq:vgt}).
Finally, we show that $\delg$ satisfies the coassociativity.
By Lemma~\mref{lem:delg}, we obtain
$$((\id\ot \delg)\delg) \circ j_A= (\id\ot \Delta_A)\Delta_A \quad\text{and}\quad((\delg\ot\id)\delg) \circ j_A= (\Delta_A\ot\id)\Delta_A. $$
By Lemma~\mref{lem:grbo},  $(\id\ot \delg)\delg$ and $(\delg\ot\id)\delg$ are \wmrb homomorphisms from $\sha_e(A)$ to $\sha_e(A)\bt\sha_e(A)\bt\sha_e(A)$. Then by the coassociativity of $\Delta_A$ and the universal property of $\sha_e(A)$, we obtain the coassociativity
$(\id\ot \delg)\delg=(\delg\ot\id)\delg.$
\end{proof}

We note that the coproduct $\delg$ satisfies the following variation of the Hochschild 1-cocycle condition.

\vsb

\begin{prop} For any $\fraka\in\sha_e(A)$, we have
\begin{equation}
\delg\big(\pg(\fraka)\big)=\Big(\pg(\fraka)+\mu\fraka\Big)\bt 1_A+(\id \bt \pg)\delg(\fraka).
\mlabel{eq:dfndeltag1}
\end{equation}
\end{prop}
\begin{proof} By Eq.~(\mref{eq:delg}), we get
\begin{eqnarray*}
\delg\big(\pg(\fraka)\big)&=&\bar{P}(\sum_{\fraka} \fraka_{(1)}\bt\fraka_{(2)})\\
&=&\sum_{\fraka} \pg(\fraka_{(1)})\bt\vg(\fraka_{(2)})1_A+\mu \fraka_{(1)}\bt\vg(\fraka_{(2)})1_A+ \fraka_{(1)}\bt\pg(\fraka_{(2)})\\
&=&\sum_{\fraka} \pg(\fraka_{(1)}\vg(\fraka_{(2)}))\bt1_A+\mu \sum_{\fraka}\fraka_{(1)}\vg(\fraka_{(2)})\bt1_A+\sum_{\fraka}\fraka_{(1)}\bt\pg(\fraka_{(2)})\\
&=&(\pg(\beta_r^{-1}(\id\ot\vg)\delg(\fraka)))\ot1_A+\mu\beta_r^{-1}(\id\ot\vg)\delg(\fraka) \bt1_A+\sum_{\fraka}\fraka_{(1)}\bt\pg(\fraka_{(2)})\\
\hspace{2cm} &=&\Big(\pg(\fraka)+\mu\fraka\Big)\bt 1_A+(\id \bt \pg)\delg(\fraka)\quad(\text{by the counicity of $\delg$}). \hspace{2.5cm}\qedhere
\end{eqnarray*}
\end{proof}

\subsection{The connectedness and the Hopf algebra structure}
\mlabel{subsec:Hopfstructure}
We now enhance the bialgebra structure on a free commutative \wmrb to a Hopf algebra structure. We first recall the following notion of filtered bialgebras (see~\mcite{GG19,KM} and the references therein).
\begin{defn}
 Let $H:=(H,m_H,\mu_H,\Delta_H,\vep_H)$ be a  bialgebra. Then
$H$ is called an (increasing) {\bf filtered bialgebra} if there are $\bfk$-submodules $H_n, n\geq 0$ of $H$ such that
\begin{enumerate}
\item
$H_n\subseteq H_{n+1}$;
\mlabel{it:subset}
\item
$H=\cup_{n\geq 0}H_n$;
\mlabel{it:cup}
\item
$H_pH_q\subseteq H_{p+q}\ \tforall p,q\geq 0$;
\mlabel{it:multip}
\item
$\Delta_H(H_n)\subseteq \bigoplus_{\substack{p+q=n}} H_p\ot H_q \ \tforall p,q,n\geq 0$;
\mlabel{it:copro}
\item
$H_n=\im \mu_H\oplus (H_n\cap \ker\vep_H).$
\mlabel{it:directsum}
\end{enumerate}
A filtered bialgebra $H$ is called {\bf connected} if $H_0=\im\, \mu_H(=\bfk 1_H)$.
\mlabel{defn:confb}
\end{defn}

Then a connected filtered bialgebra is automatically a Hopf algebra as follows. Similar results can be found in\,\mcite{KM,Ma}.

\begin{theorem}\mlabel{thm:hopf}\cite[Theorem~3.4]{GG19} Let $H:=(H,m_H,\mu_H,\Delta_H,\vep_H)$ be a connected filtered bialgebra. Then $H$ is a Hopf algebra.
\end{theorem}

Now we state the main result of this section.

\begin{theorem}
Let $A=\cup_{n\geq 0}A_n$ be a connected filtered bialgebra.  Then the free \wmrb $\sha_e(A)$ on $A$ is also a connected filtered bialgebra (and hence a Hopf algebra by Theorem~\mref{thm:hopf}).
\mlabel{thm:rbhopf}
\end{theorem}
\begin{proof}
For $0\ne a\in A$, we define the \emph{degree} of $a$ by
\begin{equation*}
\deg(a):=\min\{k\,|\,a\in A_k,k\geq 0\}.
\end{equation*}
Also for $a,b\in A$, by Definition~\mref{defn:confb}~\meqref{it:multip}, we have $ab\in A_{\deg(a)+\deg(b)}$.  Thus
\begin{equation}\mlabel{eq:abs}
	\deg(ab)\leq\deg(a)+\deg(b).
\end{equation}
By the connectedness of $A$, we get $A_0=\bfk 1_A$. Thus $\deg(1_A)=0$. Let
$$\fs:=\sha_e(A)=\bigoplus_{n\geq 1} A^{\ot n}.$$
For $0\ne \fraka=a_0\ot a_1\ot \dots\ot a_m\in A^{\ot (m+1)}\subseteq \fs$, we define
\begin{equation}
\mlabel{eq:defndeg}
\deg(\fraka):=\deg(a_0)+\deg(a_1)+\dots+\deg(a_m)+m.
\end{equation}
Denote by $\fs_k$  the  submodule generated by the set $\{\fraka\in \fs|\deg(\fraka)\le k\}$.
Then $\fs_0=A_0=\bfk 1_A$.  For $\fraka=a_0\ot\fraka'$ with $\fraka'\in A^{\ot m}$, we have
\begin{equation}
\mlabel{eq:degprop}
\deg(\fraka)=\deg(a_0)+\deg(\fraka')+1, \fraka'\in\fs_n\Longrightarrow  \fraka\in \fs_{\deg(a_0)+n+1}.
\end{equation}

We now prove that $\fs$ is a connected filtered bialgebra. The proof is divided into five parts, verifying the five conditions in Definition ~\mref{defn:confb}.
\smallskip
\item(a)
For any $\fraka\in\fs_n$, we have $\deg(\fraka)\leq n$.  Then $\fraka\in \fs_{n+1}$, proving $\fs_n\subseteq \fs_{n+1}$.
\smallskip
\item(b)
By the definition of $\fs_n$, we obtain $\fs=\cup_{n\geq 0}\fs_n$.
\smallskip
\item(c)
For any pure tensors $\fraka\in \fs_p, \frakb\in \fs_q$, we will prove $\fraka\dg \frakb\in \fs_{p+q}$ by induction on $s:=p+q$.
When $s=p+q=0$, then $p=q=0$ and so $\fraka, \frakb\in \fs_0=\bfk 1_A$. Thus
$$\fraka\dg\frakb=\fraka\frakb\in \bfk1=\fs_0=\fs_{p+q}.$$
If $p=0$ or $q=0$,  then $\fraka\in \fs_0$ or $\frakb\in \fs_0$. Thus $\fraka\dg\frakb=\fraka\frakb\in \fs_{p}$ or $\fs_q$ since $\fs_0=\bfk1_A$, and so $\fraka\dg\frakb\in\fs_{p+q}$.
Let $p,q\geq 1$. Write $\fraka=a_0\ot\fraka'$ with $\fraka'\in A^{\ot m}$ and $\frakb=b_0\ot\frakb'$ with $\frakb'\in A^{\ot n}$. For a given $k\geq 2$, assume $\fraka\dg \frakb\in \fs_{p+q}$ holds for $s=p+q\le k$ and now consider the case when $s=p+q=k+1\geq 3$. By Eq.~(\mref{eq:dfndg}), we have
\begin{equation}\mlabel{eq:pqs}
\fraka\dg\frakb=a_0b_0\ot\big(\fraka'\dg(1_A\ot\frakb')\big)+a_0b_0\ot\big((1_A\ot\fraka')\dg\frakb'\big)+\lambda a_0b_0\ot(\fraka'\dg\frakb')+\kappa a_0b_0(\fraka'\dg\frakb').
\end{equation}
Consider the first term on the right hand side of Eq.~(\mref{eq:pqs}). By Eq.~(\mref{eq:degprop}), we have
$$\deg(\fraka')=\deg(\fraka)-\deg(a_0)-1\le {p-\deg(a_0)-1}$$
and
$$\deg(\frakb')=\deg(\frakb)-\deg(b_0)-1\le {q-\deg(b_0)-1}.$$
So $\fraka'\in \fs_{p-\deg(a_0)-1},  \frakb'\in \fs_{q-\deg(b_0)-1}$ and $1_A\ot\frakb'\in \fs_{q-\deg(b_0)}$.
By the induction hypothesis and Eq.~(\mref{eq:degprop}), we obtain
$\fraka'\dg(1_A\ot\frakb')\in \fs_{p+q-\deg(a_0)-\deg(b_0)-1}$. Then by Eqs.~(\mref{eq:abs}) and~(\mref{eq:degprop}), we obtain
$$a_0b_0\ot\big(\fraka'\dg(1_A\ot\frakb')\big)\in \fs_{p+q}.$$
Applying the same method, we can show that $a_0b_0\ot\big((1_A\ot\fraka')\dg\frakb'\big), a_0b_0\ot(\fraka'\dg\frakb') \in \fs_{p+q}$.
We now consider the fourth term on the right hand side of Eq.~(\mref{eq:pqs}). Since
$a_0b_0(\fraka'\dg \frakb')=(a_0b_0\fraka')\dg \frakb'$ and $\deg(a_0b_0\fraka')\leq \deg(a_0)+\deg(b_0)+\deg(\fraka')$ by Eq.~\meqref{eq:abs},
we have $a_0b_0\fraka'\in \fs_{p+\deg(b_0)-1}$, and so $\deg(a_0b_0\fraka')+\deg(\frakb')\leq p+q-2$, giving $a_0b_0(\fraka'\dg\frakb')\in \fs_{p+q}$ by the induction hypothesis.
Thus $\fraka \dg \frakb \in \fs_{p+q}$, proving $ \fs_p\fs_q\subseteq \fs_{p+q}$.
\smallskip

\item(d)
To prove $\delg(\fs_n)\subseteq \bigoplus_{\substack{p+q=n}} \fs_p\bt \fs_q \ \tforall n\geq 0$, we use the induction on $n\geq 0$.  When $n=0$, then $\fs_0=\bfk 1_A$. Since $\delg(1_A)=\da(1_A)=1_A\ot 1_A$, we have $\delg(\fs_0)\subseteq \fs_0\bt\fs_0$.
For a given $n\geq 0$, assume $\delg(\fs_\ell)\subseteq \bigoplus_{\substack{p+q=\ell}} \fs_p\bt \fs_q$ holds for all $\ell\leq n$.  Consider $\fraka \in \fs_{n+1}$. Write $\fraka=a_0\ot\fraka'\in A^{\ot (k+1)}$ with $k\geq 0$. If $k=0$, then $\fraka=a_0\in A_{n+1}$. Then
\vsb
 $$\delg(\fraka)=\da(a_0)\subseteq\bigoplus_{p+q=n+1}A_{p}\bt A_{q}\subseteq \bigoplus_{p+q=n+1} \fs_p\bt\fs_q.
 \vsb
 $$
If $k\geq 1$, then
\vsc
\begin{align*}
\delg(\fraka)={}& \delg\big(a_0\dg \pg(\fraka')\big)\\
={}& \da(a_0)\bug\delg\big(\pg(\fraka')\big)\\
={}& \da(a_0)\bug\Big(\pg(\fraka')\bt1_A+(\id\bt\pg)\delg(\fraka')+\mu\fraka'\bt1_A\Big).
\end{align*}
By $\deg(\fraka')= \deg(\fraka)-\deg(a_0)-1\leq n-\deg(a_0)$ and the induction hypothesis, we have
\begin{equation*}
\delg(\fraka') \in \bigoplus_{\substack{s+t=\deg(\fraka')}} \fs_s \bt \fs_t.
\end{equation*}
By Eq.~(\mref{eq:degprop}), we have $\deg(1_A\ot\frakb)=\deg(\frakb\ot 1_A)=\deg(\frakb)+1$ for all $\frakb\in A^{\ot k}$. Thus
\begin{equation*}
\deg(\pg(\frakb))=\deg(\frakb)+1.
\end{equation*}
Then
\vsb
$$(\id\bt\pg)\delg(\fraka')\in  \bigoplus_{\substack{s+t=\deg(\fraka')}} \fs_s \bt \pg\big(\fs_t\big)
\subseteq \bigoplus_{\substack{s+t=\deg(\fraka')}} \fs_s \bt \fs_{t+1}
\subseteq \bigoplus_{\substack{p_2+q_2=\deg(P(\fraka'))}}\fs_{p_2}\bt\fs_{q_2}
$$
by setting $p_2:=s$ and $q_2:=t+1$.
So we have
$$\pg(\fraka')\bt1_A+(\id\bt\pg)\delg(\fraka')+\mu\fraka'\bt1_A\in \bigoplus_{\substack{p_2+q_2=\deg(\pg(\fraka'))}} \fs_{p_2}\bt \fs_{q_2}.$$
Thus we get
\vsb
\begin{align*}
\delg(\fraka)&=\da(a_0)\bug\big(\pg(\fraka')\bt1_A+(\id\bt\pg)\delg(\fraka')+\mu\fraka'\bt1_A\big)\\
&\in\bigoplus_{\substack{p_1+q_1=\deg(a_0)}}\fs_{p_1}\bt\fs_{q_1}\,\,\bug\bigoplus_{\substack{p_2+q_2=\deg(\pg(\fraka'))}}\fs_{p_2}\bt \fs_{q_2}\\
&\subseteq{} \bigoplus_{\substack{p_1+q_1=\deg(a_0)\\p_2+q_2=\deg(\pg(\fraka'))}} \Big(\fs_{p_1}\dg \fs_{p_2}\Big) \bt \Big(\fs_{q_1}\dg\fs_{q_2}\Big)\\
&\subseteq \bigoplus_{\substack{p_1+q_1+p_2+q_2=\deg(\fraka)}} \fs_{p_1+p_2} \bt \fs_{q_1+q_2}\quad(\text{by Eq.~(\mref{eq:degprop})})\\
 &\subseteq \bigoplus_{\substack{p+q=n+1}}\fs_p \bt \fs_q \quad(\text{by setting $p:=p_1+p_2$ and $q:=q_1+q_2$}).
\end{align*}
This completes the induction and thus proves that $\delg(\fs_n)\subseteq  \bigoplus_{\substack{p+q=n}} \fs_p \bt \fs_q$.
\smallskip

\item(e)
Let $e=\mg \vg$. Then by $\vg(1_A)=1_\bfk$ and $\mg(1_\bfk)=1_A$,  we know that $e(1_A)=1_A$ and $e(\fraka)=0$ for every pure tensor $\fraka \in \ker\vg$. This gives $\vg\mg=\id_\bfk$, proving that $e^2=e$, $\mg$ is injective and $\vg$ is surjective. Thus
\vsb
\begin{equation}
\fs=\im e \oplus \ker e=\im \mg \oplus \ker \vg.
\mlabel{eq:fsvg}
\end{equation}
We now prove that $\fs_n=\im \mg \oplus \big(\fs_n \cap \ker\vg\big)$ for $n\geq 0$. When $n=0$, by Eq.~(\mref{eq:fsvg}) and $\fs_0=\bfk 1_A=\im\mg$, we have
\vsb
$$\im \mg\oplus\big(\fs_0 \cap \ker\vg\big)=\im \mg \oplus\big(\im \mg \cap \ker\vg\big)=\im \mg=\fs_0.$$
When $n\geq 1$, we get $\deg(\fraka)\le n$ for all $\fraka \in \fs_n$. Write $\fraka=\vg(\fraka) 1_A+\fraka-\vg(\fraka) 1_A$. Then
$$\vg\big(\fraka-\vg(\fraka)1_A\big)=\vg(\fraka)-\vg\big(\vg(\fraka)1_A\big)=\vg(\fraka)-\vg(\fraka)\vg(1_A)=\vg(\fraka)-\vg(\fraka)=0.$$
Hence $\fraka-\vg(\fraka)1_A \in \ker\vg$. Since $\fs_0 \subseteq \fs_n$ by Item (\mref{it:subset}), we have $\fraka-\vg(\fraka) 1_A\in \fs_n$. So $\fraka-\vg(\fraka)1_A \in \fs_n\cap \ker\vg$ and then $\fs_n=\im \mg + \big(\fs_n \cap \ker\vg\big)$.
Since
$$\big(\fs_n \cap \ker\vg\big)\cap \im \mg=\fs_n \cap \big(\ker\vg \cap \im \mg\big)=\fs_n \cap 0=0,$$
 we get $\fs_n=\im \mg \oplus \big(\fs_n \cap \ker\vg\big), n\geq 1$.
Thus $$\fs_n=\im \mg \oplus \big(\fs_n \cap \ker\vg\big), n\geq 0.$$
Then by Definition~\mref{defn:confb}, $\fs=\sha_e(A) $ is a connected filtered bialgebra.
\end{proof}

We end the paper by revisiting the special case of $\sha_e(A)$ when $A$ takes the base ring $\bfk$ and make explicit the coproduct on this Hopf algebra. Recall that
\begin{equation*}
\sha_e(\bfk)=\bigoplus_{\substack{n\geq 0}} \bfk^{\ot (n+1)}=\bigoplus_{\substack{n\geq 0}} \bfk\bfone^{\ot (n+1)}.
\end{equation*}
Note that $\bfk$ is a connected  filtered $\bfk$-bialgebra with the coproduct and counit given by
$$\de_\bfk:\bfk\to \bfk\ot \bfk,\ \ x\mapsto x\ot \bfone ; \quad \ve_\bfk:=\id_\bfk:\bfk \to \bfk,\ \ x\mapsto x ,\ \ \quad x \in \bfk,$$
and the filtration given by $\bfk_0=\bfk$ and $\bfk_n=\bfk, n\geq 1$.
Then the coproduct $\delg$ on $\sha_e(\bfk)$ becomes the recursion
\vsb
\begin{equation*}
\delg:\sha_e(\bfk)\to \sha_e(\bfk)\bt\sha_e(\bfk),\quad
  \bfone^{\ot (k+1)}\mapsto \left\{\begin{array}{lll}\bfone\bt \bfone,\quad &k=0, \\
\pg(\bfone^{\ot k})\bt \bfone+(\id\bt \pg)\delg(\bfone^{\ot k})+\mu\bfone^{\ot k}\bt \bfone,&k\geq 1,\end{array}\right.
\end{equation*}
and the counit $\vg$ becomes
\vsb
$$\vg:\sha_e(\bfk)\to \bfk,\ \ \bfone^{\ot (k+1)}\mapsto (-\mu)^k\bfone, \ \ k\geq 0.$$

\begin{prop}With the above notations, we have
\vsb
\begin{equation}
\delg(\bfone^{\ot (k+1)})=\sum_{i=0}^{k}\bfone^{\ot (i+1)}\bt\bfone^{\ot (k+1-i)}+\mu\sum_{i=0}^{k-1}\bfone^{\ot (i+1)}\bt\bfone^{\ot (k-i)},\quad\, k\geq 1.\mlabel{eq:dgk}
\end{equation}
\end{prop}
\begin{proof}
With $\delg(\bfone^{\ot \bfone})=\bfone\bt \bfone$ as required, we prove Eq.~(\mref{eq:dgk}) by induction on $k\geq 1$.  When $k=1$, we have
\begin{align*}
\delg(\bfone^{\ot 2})={}& \pg(\bfone)\bt \bfone+(\id\bt \pg)\delg(\bfone)+\mu\bfone\bt \bfone\\
={}& \bfone^{\ot 2}\bt\bfone+(\id\bt \pg)(\bfone\bt \bfone)+\mu\bfone\bt \bfone\\
={}& \bfone^{\ot 2}\bt\bfone+\bfone\bt \bfone^{\ot 2}+\mu\bfone\bt \bfone\\
={}& \sum_{i=0}^{1}\bfone^{\ot (i+1)}\bt\bfone^{\ot (1+1-i)}+\mu\sum_{i=0}^{0}\bfone^{\ot (i+1)}\bt\bfone^{\ot (1-i)}.
\end{align*}
Assume that Eq.~(\mref{eq:dgk}) holds for $k\geq 1$ and consider the case for $k+1$. Then
\begin{align*}
&\delg(\bfone^{\ot (k+2)})\\
={}& \pg(\bfone^{\ot (k+1)})\bt \bfone+(\id\bt\pg)\delg(\bfone^{\ot (k+1)})+\mu\bfone^{\ot (k+1)}\bt \bfone\\
={}& \bfone^{\ot (k+2)}\bt \bfone+(\id\bt\pg)\Big(\sum_{i=0}^{k}\bfone^{\ot (i+1)}\bt\bfone^{\ot (k+1-i)}+\mu\sum_{i=0}^{k-1}\bfone^{\ot (i+1)}\bt\bfone^{\ot (k-i)}\Big)+\mu\bfone^{\ot (k+1)}\bt \bfone\\
={}& \bfone^{\ot (k+2)}\bt \bfone+\sum_{i=0}^{k}\bfone^{\ot (i+1)}\bt\bfone^{\ot (k+2-i)}+\mu\sum_{i=0}^{k-1}\bfone^{\ot (i+1)}\bt\bfone^{\ot (k+1-i)}+\mu\bfone^{\ot (k+1)}\bt \bfone\\
={}& \sum_{i=0}^{k+1}\bfone^{\ot (i+1)}\bt\bfone^{\ot (k+2-i)}+\mu\sum_{i=0}^{k}\bfone^{\ot (i+1)}\bt\bfone^{\ot (k+1-i)}.
\end{align*}
This completes the induction and hence the proof of Eq.~(\mref{eq:dgk}).
\end{proof}
\smallskip

\noindent {\bf Acknowledgements}: This work is supported by the National Natural Science Foundation of
China (11961031, 12326324) and Jiangxi Provincial Natural Science Foundation (20224BAB201003).

\noindent
{\bf Declaration of interests. } The authors have no conflicts of interest to disclose.

\noindent
{\bf Data availability. } Data sharing is not applicable as no new data were created or analyzed.

\end{document}